\documentclass[reqno]{amsart}
\usepackage{amsmath,amsbsy,amsfonts}
\usepackage{color}
\date{}
\def\nd{\noindent}
\def\thend{\rule{3mm}{3mm}}

\newtheorem{theorem}{Theorem}[section]

\newtheorem{prop}{Proposition}[section]
\newtheorem{lem}{Lemma}[section]
\newtheorem{rmk}{Remark}[section]

\newenvironment{dem}{\nd{\bf Proof. }}{\hskip.3cm\thend}

\newcommand{\w}{W_0^{1,\Phi}(\Omega)}

\newcommand{\cqd}{\hspace{10pt}\fbox{}}

\newcommand{\Int}{\displaystyle\int_{\Omega}}

\newcommand{\ds}{\displaystyle}
\newcommand{\Fr}{\displaystyle\frac}

\newcommand{\on}{o_n(1)}
\newcommand{\el}{\ell^*}

\newcommand{\R}{\mathbb{R}}

\newcommand{\N}{\mathcal{N}_\lambda}
\begin{document}
%\color{red}

\title[Concave-convex effects for critical quasilinear elliptic problems]
{Concave-convex effects for critical quasilinear elliptic problems}
\vspace{1cm}
%%%%%%%%%%%%%%%%%%%%%%%%%%%%%%%%%%%%%%%%%%%%%%%%%%%%%%%%%%%%%%%%%%%%%%%%

\author{M. L. M. Carvalho}
\address{M. L. M. Carvalho \newline Universidade Federal de  Goi\'as, IME, Goi\^ania-GO, Brazil }
\email{\tt marcos$\_$leandro$\_$carvalho@ufg.br}

\author{E. D. da Silva}
\address{Edcarlos D da Silva \newline  Universidade Federal de  Goi\'as, IME, Goi\^ania-GO, Brazil}
\email{\tt edcarlos@ufg.br}

\author{J. V. Goncalves.}
\address{J. V. Goncalves. \newline Universidade Federal de  Goi\'as,IME, Goi\^ania-GO, Brazil}
\email{\tt goncalvesjva@ufg.br}

\author{C. Goulart}
\address{C. Goulart \newline Universidade Federal de Goi\'as, Regionsl Jata\'i, Jata\'i, Brazil }
\email{\tt claudiney@ufg.br}

\thanks{The third author was partially supported by Cnpq and Fapego grants.}

\begin{abstract}
It is established existence, multiplicity and asymptotic behavior of positive solutions for a quasilinear elliptic problem driven by the $\Phi$-Laplacian operator. One of these solutions is obtained as ground state solution by applying the well known Nehari method. The semilinear term  in the quasilinear equation is a concave-convex function which presents a critical behavior at infinity. The concentration compactness principle is used in order to recover the compactness required in variational methods.
\end{abstract}

\maketitle

\section{Introduction}

In this work we deal with existence, multiplicity and asymptotic behaviour  of positive solutions of the  problem
\begin{equation}\label{eq1}
-\Delta_{\Phi} u   =  \lambda a(x)|u|^{q-2}u  + b(x)|u|^{\ell^*-2}u~   \mbox{ in }\    \Omega,~~u  = 0~   \mbox{ in } \partial\Omega,
\end{equation}
where $\Omega\subset\mathbb{R}^{N}$ is a bounded smooth domain, $\lambda>0$ is a parameter, $\ell^* :=  N \ell/(N - \ell)$ with $1 < \ell < N$ and $a, b: \Omega \to \mathbb{R}$ are  functions which may change sign. The operator $\Delta_{\Phi}$ is named $\Phi$-Laplacian and is given by
$$
\Delta_{\Phi} u =  \mbox{div} ( \phi(|\nabla u|) \nabla u)
$$
\nd where $\phi: (0,\infty)\rightarrow (0,\infty)$ is a $C^{2}$-function satisfying
\begin{description}
  \item[$(\phi_1)$] $\displaystyle \lim_{s \rightarrow 0} s \phi(s)= 0, \displaystyle \lim_{s \rightarrow \infty} s \phi(s)= \infty$;
  \item[$(\phi_2)$] $s\mapsto s\phi(s)$ is strictly increasing.
 \end{description}
\nd We extend $s \mapsto s \phi(s)$ to  $\mathbb{R}$ as an odd function. The function $\Phi$ is given by
$$
\Phi(t) = \int_{0}^{t} s \phi(s) ds,~ t \geq 0
$$
\nd and satisfies $\Phi(t) = \Phi(-t)$ for each $t \leq 0$. For futher results on Orlicz and Orlicz-Sobolev framework we refer the reader to Adans \cite{A}, Fukagai et al \cite{Fuk_1,Fuk_2}, Gossez \cite{Gz1,gossez-Czech} and Rao \cite{Rao1}.
\vskip.1cm

\nd Quasilinear elliptic problems such as \eqref{eq1} have been considered in order to explain many physical problems which arise from
 Nonlinear Elasticity, Plasticity and both Newtonian and  Non-Newtonian Fluids.  We refer the reader to \cite{appl1,appl2,Fuk_1,Fuk_2,fang}.
\vskip.2cm

\nd When $\phi := 2$, $a = b := 1$ we notice that $\ell = 2$. Then problem \eqref{eq1} reads as
 \begin{equation}\label{eq1-0}
-\Delta u = \lambda |u|^{q-2}u + |u|^{2^{*}-2}u  \;\; in \;\; \Omega,~ u = 0 \;\; on \;\; \partial \Omega.
 \end{equation}

\nd In the pioneering paper \cite{BN-1}, Br\'ezis \& Nirenberg proved  results on existence of positive solutions of \eqref{eq1-0}. A new  variational technique was developed  to overcome difficulties due to the presence of the critical Sobolev exponent  $2^{*} = \frac{2N}{N - 2}$.
\vskip.1cm

\nd  Problem \eqref{eq1-0} was later addressed  by Ambrosetti, Br\'ezis \& Cerami \cite{ABC-1} where among other results it was shown that there is some $\Lambda > 0$ such that \eqref{eq1-0}
has
$$
\mbox{a positive minimal solution}
\;\; u_{\lambda} \in H_{0}^{1}~ \mbox{for each}~  \lambda \in (0, \Lambda),
\;\; \mbox{with} \leqno(i)
$$
$$
\frac{1}{2} \int_{\Omega} \vert \nabla u_{\lambda} \vert^{2} -
\frac{\lambda}{q + 1} \int_{\Omega} u_{\lambda}^{q + 1} -
\frac{1}{2^{*}} \int_{\Omega} u_{\lambda}^{2^{*}} \; < \; 0
$$
$$
\mbox{when} \;\; 1 < q < 2, ~N \geq 1,
$$
$$
\mbox{a positive weak solution} \;\; u_{\lambda} \in H_{0}^{1}
 \;\; \mbox{for} \;\; \lambda = \Lambda  \leqno(ii)
$$
$$
\mbox{when} \;\;
2 < q < 3, ~ N \geq 3,
$$
$$
\mbox{no positive solution when} \;\; \lambda > \Lambda. \leqno(iii)
$$

\nd Moreover, in the first case above, $\Vert u_{\lambda}
\Vert_{\infty} \rightarrow 0$ as $\lambda \rightarrow 0$.
\vskip.1cm

\nd We further refer the reader   Alama \& Tarantello \cite{Alama-Tarantello-1},
Admurthi, Pacella \& Yadava \cite{Ad-Pacella-Yadava-1}  and their references.
\vskip0.3cm

\nd When $\phi(t) = r t^{r-2}$,  $1< r < \infty$ and $a = b := 1$
problem \eqref{eq1} becomes
\begin{equation}\label{ar}
\displaystyle  - \Delta_r u = \lambda |u|^{q-2}u  + |u|^{p-2}u ~ \mbox{in}~
\Omega,~~ u = 0~  \mbox{on}~ \partial \Omega.
\end{equation}
\nd This problem was studied by Ambrosetti, Garcia Azorero \& Peral \cite{ambrosetti-azorero-peral-1} and subsequently by many other researchers.
\vskip.1cm

\nd It is worthwhile mention that conditions $(\phi_{1})- (\phi_{2})$ implies that the function $\Phi$ is an N-function and in addition due to the expression of $\Delta_{\Phi}$ it is natural to work in the framework of Orlicz-Sobolev spaces, for basic results on Orlicz and Orlicz Sobolev spaces we infer the reader to \cite{Gz1,gossez-Czech,Rao1}. It is well known that $W^{1,\Phi}_{0}(\Omega)$ is not equal in general to $W^{1,q}_{0}(\Omega)$ for any $q \in [1, + \infty)$. As example we cite $\Phi(t) = |t|^{p} ln (1 + |t|), p > 1,$ which satisfies $W^{1,\Phi}_{0}(\Omega) \neq W^{1,q}_{0}(\Omega)$ for any $q \in [1, + \infty)$. Hence is not possible to consider the usual Sobolev spaces $W^{1,q}_{0}(\Omega)$ in order to ensure existence and multiplicity of solutions for the problem \eqref{eq1}.

The following additional condition on $\phi$ will also be assumed:

\begin{description}
 \item[($\phi_3$)] $ -1<\ell-2:=\ds\inf_{t>0}\Fr{(t\phi(t))''t}{(t\phi(t))'}\leq \ds\sup_{t>0}\Fr{(t\phi(t))''t}{(t\phi(t))'}=:m-2<N-2.$
\end{description}

The reader is  referred to  \cite{A,Rao1} regarding Orlicz-Sobolev spaces.  The usual norm on $L_{\Phi}(\Omega)$ is ( Luxemburg norm),
$$
\|u\|_\Phi=\inf\left\{\lambda>0~|~\int_\Omega \Phi\left(\frac{u(x)}{\lambda}\right) dx \leq 1\right\}
$$
and  the  Orlicz-Sobolev norm of $ W^{1, \Phi}(\Omega)$ is
\[
  \displaystyle \|u\| = \|u\|_\Phi+\sum_{i=1}^N\left\|\frac{\partial u}{\partial x_i}\right\|_\Phi.
\]
We say that a N-function $\Psi$ grow essentially more slowly than $\Phi_*$, we write $\Psi<<\Phi_*$ whenever
$$
\lim_{t\rightarrow \infty}\frac{\Psi(\lambda t)}{\Phi_*(t)}=0,~~\mbox{for all}~~\lambda >0.
$$
Recall that
$$
\widetilde{\Phi}(t) = \displaystyle \max_{s \geq 0} \{ts - \Phi(s) \},~ t \geq 0.
$$

The imbedding below (cf. \cite{A, DT}) will be used in this paper:
$$
\displaystyle W_{0}^{1,\Phi}(\Omega) \stackrel{\tiny cpt}\hookrightarrow L_\Psi(\Omega),~~\mbox{if}~~\Psi<<\Phi_*,
$$
in particular, as $\Phi<<\Phi_*$ (cf. \cite[Lemma 4.14]{Gz1}),
$$
W_{0}^{1,\Phi}(\Omega) \stackrel{\tiny{cpt}} \hookrightarrow L_\Phi(\Omega).
$$
Furthermore, we have the following embeddings
$$
W_0^{1,\Phi}(\Omega) \stackrel{\mbox{\tiny cont}}{\hookrightarrow} L_{\Phi_*}(\Omega)
$$
and
$$
L_{\Phi}(\Omega) \stackrel{\mbox{\tiny cont}}{\hookrightarrow} L^{\ell}(\Omega),  L_{\Phi_{*}}(\Omega) \stackrel{\mbox{\tiny cont}}{\hookrightarrow} L^{\ell^{*}}(\Omega).
$$

Under assumptions $(\phi_{1}) - (\phi_{3})$ it turns out that  $\Phi$ and $\widetilde{\Phi}$  are  N-functions  satisfying  the $\Delta_2$-condition, (cf. \cite[p 22]{Rao1}).

\begin{rmk}\label{conseqphi3}
Under assumption $(\phi_{3})$ we observe that
\begin{equation}
\ell-2\leq\Fr{\phi'(t)t}{\phi(t)}\leq m-2,\ \
 \ell\leq\Fr{\phi(t)t^2}{\Phi(t)}\leq m, t > 0.
\end{equation}
Moreover, we have that
$$\left\{\begin{array}{rcl}
t^2\phi''(t)&\leq& (m-4)t\phi'(t)+(m-2)\phi(t)\\
t^2\phi''(t)&\geq& (\ell-4)t\phi'(t)+(\ell-2)\phi(t), t \geq 0.
\end{array}\right.$$
\end{rmk}

\nd Under conditions $(\phi_{1})- (\phi_{2})$ and $(\phi_{3})$ the Orlicz-Sobolev space $W^{1,\Phi}_{0}(\Omega)$ is Banach and reflexive with respect to the  standard norm denoted $\Vert . \Vert$.
\vskip.1cm

\nd We also point out that $\phi(t) = 2$ and $\phi(t) = rt^{r -2}$ satisfy $(\phi_{1})- (\phi_{2})$ and $(\phi_{3})$.
\vskip.2cm

\nd Moreover, when $\phi(t) = 2$ then $m = \ell= 2$,
$\Delta_{\Phi} = \Delta$ and $W^{1,\Phi}_{0}(\Omega) = H_0^1(\Omega)$. When $\phi(t) = rt^{r -2}$ then $m = \ell= r$, $\Delta_{\Phi} = \Delta_r$ and $W^{1,\Phi}_{0}(\Omega) = W_0^{1,r}(\Omega)$.
\vskip.1cm

\nd Many other well known operators are examples of $\Delta_{\Phi}$. For instance, if $\phi(t) = p_{1}t^{p_{1} -2} + p_{2}t^{p_{2} -2}$ with $ 1 < p_{1} < p_{2} < \infty$ then $\phi$ satisfies hypotheses $(\phi_{1})-(\phi_{2})$ and the operator in problem \eqref{eq1} reads as $-\Delta_{p_{1}} u - \Delta_{p_{2}} u$ which is known as the
$(p_{1},p_{2})$-Laplacian and was extensively studied in the last years, see \cite{MugnaiPapageorgiou,tanaka}. We mention that in this case $\ell = p_{1}$ and $m = p_{2}$.

\nd Another class of operators is the so called anisotropic elliptic problem included here as example for $\Delta_{\Phi}$ is obtained by setting
$$\phi(t) = \displaystyle\sum_{j=1}^{N} t^{p_{j}-2}, \Phi(t) =  \displaystyle\sum_{j=1}^{N} \dfrac{t^{p_{j}}}{p_{j}} $$
where $1 < p_{1} < p_{2} < \ldots < p_{N} < \infty$ and
\begin{equation}
p^{\star} = \displaystyle \dfrac{N} {\left(\sum_{j =1}^{N}\frac{1}{p_{j}}\right) - 1},  \displaystyle\sum_{j=1}^{N} \dfrac{1}{p_{j}} > 1.
\end{equation}
\nd Here we consider the case $ p_{N} < p^{\star}$ and $\overline{p} = \dfrac{N}{\sum_{j =1}^{N}\frac{1}{p_{j}}}$ is the mean harmonic for the numbers $p_{j}$ with $j = 1, 2, \ldots, N$. This number satisfies $p^{\star} = \frac{N \overline{p}}{N - \overline{p}}$. It is no hard to verify that hypotheses $(\phi_{1})- (\phi_{3})$ are satisfied for the anisotropic elliptic problem. This operator have been considered during the lasts years which has a rich physical motivation, see \cite{Mot1,Mot2}. For further references we refer the reader to \cite{Fra,Gio,Mi1,Mi2,tru} and references therein.
\vskip.4cm

\nd It is important to emphasize that a great interest on problem \eqref{eq1} for the Laplacian operator have been made since the seminal paper of
Ambrosetti and Rabinowitz \cite{AR}. Our main purpose in this work is to guarantee existence and multiplicity of solutions
for quasilinear elliptic equations drive by $\Phi$-Laplacian using indefinite concave-convex nonlinearities. More specifically, we shall
consider problem \eqref{eq1} where the functions $a$ and $b$ changes sign.

\nd The main aim in this work is to consider the critical growth in the problem \eqref{eq1}. Elliptic problems with critical nonlinearities have been widely considered since the celebrated works of Lions \cite{lions1,lions2,lions3,lions4}. For quasilinear elliptic problems we infer the reader to
\cite{G,AGPeral,Fuk_1,Radu,GPeral,M,S,Willem} and references therein. The main difficult here is the loss of compactness for the embedding
$\w$ into $L^{\ell^{\star}}(\Omega)$. In order to overcome this this difficult we apply the concentration compactness principle together variational methods ensuring our main results.

\nd In this paper we shall assume the following set of technical conditions:
\begin{flushleft}
$(H)$ $\hspace*{.7cm}$ $1<q<\Fr{\ell(\el-m)}{\el-\ell}\leq \ell\leq m<\el,$\hspace*{.4cm}$ a,b \in L^\infty(\Omega),~~~ a^+,b^+\not\equiv 0.$
\end{flushleft}
\vskip.2cm
 The main feature in this work is to use the Nehari method in order to achieve our main results. The hypothesis $(H)$ is essential for the
 minimization procedure which shows that the critical value on the Nehari manifold is negative, see Section 3 ahead.

\nd We recall that under $(\phi_{1})-(\phi_{3})$ the functional $J_{\lambda} : W^{1, \Phi}_{0}(\Omega) \rightarrow \mathbb{R}$ given by
$$
J_{\lambda}(u)=\Int\Phi(|\nabla u|)-\Fr{\lambda}{q} \Int a(x)|u|^{q}-\Fr{1}{\el} \Int b(x)|u|^{\el},\ \ u\in W_0^{1,\Phi}(\Omega),
$$
is well-defined and is of class $C^{1}$. Actually, the derivative of $J_{\lambda}$ is given by
\begin{equation}\label{functional}
\left<J'_{\lambda}(u),v\right> = \Int\phi(|\nabla u|)\nabla u\nabla v-\lambda \Int a(x)|u|^{q-2}uv- \Int b(x)|u|^{\el-2}uv\nonumber
\end{equation}
for any $u,v\in W_0^{1,\Phi}(\Omega)$.  Hence finding weak solutions for the problem \eqref{eq1} is equivalent to find critical points for the functional $J_{\lambda}$.
In general, under hypotheses $(\phi_{1})-(\phi_{3})$, the functional $J_{\lambda}$ is not of class $C^{2}$.
\vskip.1cm

\nd A weak solution $u \in W^{1,\Phi}_{0}(\Omega)$ for equation \eqref{eq1} is said to be a ground state solution when $u$ is a minimal energy solution in the set of all nontrivial solutions. In this work we shall prove existence nonnegative ground state solution using the Nehari method. Besides that, we find another nonnegative solution for the problem \eqref{eq1} using a minimization procedure.  An overview on this subject can be found in Szulkin \verb"&"  Weth \cite{ST,SW}.
\vskip.1cm

\nd Quasilinear elliptic problems driven by $\Phi$-Laplacian operator have been extensively discussed during the last years. We refer the reader to the important works \cite{claudianor,JVMLED, JME, charro,chung,hui-2,MugnaiPapageorgiou,fang}.
\vskip.1cm

\nd  In \cite{JVMLED}
the authors considered existence of positive solutions for quasilinear elliptic problems where the nonlinear term is superlinear at infinity. In \cite{chung,fang} the authors studied existence and multiplicity of solutions where the nonlinear term is also superlinear.  In \cite{claudianor} was studied the critical case using the well known concentration-compactness argument.
\vskip.1cm

\nd Regarding  concave-convex nonlinearities we further refer the reader to  \cite{wu2,DJ,DJ2,depaiva,wu,wu3}.
\vskip.1cm

\nd It is worthwhile mentioning that in our main theorems the functions $a, b$ may change sign and no homogeneity conditions either on the operator or on the nonlinear term is required. More specifically, we emphasize that our nonlinear operator $\Delta_{\Phi}$ is not homogeneous which is a serious difficult
in elliptic problems.   To the best of our knowledge, there is no result on elliptic problems with concave-convex functions for the  $\Phi$-Laplacian  operator in the critical case.
\vskip.1cm

Our  main results are stated below.

\begin{theorem}\label{teorema1} Suppose $(\phi_{1}) - (\phi_{3})$ and $(H)$. Then there exists $\Lambda_1> 0$ such that for each $\lambda \in (0,  \Lambda_{1})$,  problem \eqref{eq1} admits at least one nonnegative ground state solution $u_{\lambda}$ satisfying $J_{\lambda}(u_{\lambda}) < 0$ and $\displaystyle \lim_{\lambda \rightarrow 0^{+}} \|u_{\lambda}\| = 0$.
\end{theorem}

Now we shall state our second result.

\begin{theorem}\label{teorema2}
	Suppose $(\phi_{1}) - (\phi_{3})$ and $(H)$. Then there exists $\Lambda_2 > 0$ in such way that for each $\lambda \in (0,  \Lambda_{2})$, problem \eqref{eq1} admits at least one nonnegative weak solution $v_{\lambda}$ satisfying $J_{\lambda}(v_{\lambda}) > 0$ for any $0 < \lambda < \Lambda_{2}$.
\end{theorem}

\nd As a consequence of the  results just above we have the following multiplicity result.

\begin{theorem}\label{teorema3}
	Suppose $(\phi_{1}) - (\phi_{3})$ and $(H)$. Set $\Lambda =  \min \{\Lambda_1, \Lambda_2 \}$.  Then  for each $\lambda \in (0,  \Lambda)$, problem \eqref{eq1} admits at least two nonnegative weak solutions $u_{\lambda}, v_{\lambda} \in \w$ satisfying $ J_{\lambda}(u_{\lambda}) < 0 < J_{\lambda}(v_{\lambda})$. Furthermore, the function $u_{\lambda}$ is a ground state solution for each $\lambda \in (0,\Lambda)$.
\end{theorem}

\nd In order to achieve our results we shall consider the Nehari manifold $\mathcal{N}_{\lambda}$ introduced in \cite{Nehari}. Here we also refer to \cite{Brown2,Brown1,Drabek,ST,SW} where the authors establish a precise description on the fibering maps.
\vskip.1cm

\nd A main point during this work is that due to the concave-convex nonlinearities present in \eqref{eq1}, the
Ambrosetti-Rabinowitz condition is not satisfied in general. Furthermore, when  $a,b$ are  functions that change sign, the well known nonquadraticity condition introduced by Costa-Magalh\~aes \cite{CM} does not work anymore. Those conditions are used to prove that certain Palais-Smale sequences are  bounded. In order to overcome this difficulty  we shall employ the Nehari manifold method.
\vskip.1cm

\nd In this work we employ the fibering maps, (which thanks to $(\phi_{1})- (\phi_{3})$ are of class $C^{2}$),
to split the Nehari manifold into two parts say $\mathcal{N}_{\lambda}= \mathcal{N}_{\lambda}^{+} \cup \mathcal{N}_{\lambda}^{-}$.
More specifically, in order to achieve our results we shall consider the Nehari manifold $\mathcal{N}_{\lambda}$ introduced in \cite{Nehari}. Here we also refer \cite{Brown2,Brown1,Drabek,ST,SW} where the authors establish a precisely description on the fibering maps. In the present work the main difficult is that $a$ and $b$ does not have defined sign, i.e, the functions $a,b$ can be change signs. Furthermore, the nonlinear operator $\Delta_{\Phi}$ is not homogeneous.  In order to overcome these difficulties we split the Nehari manifold into two parts $\mathcal{N}_{\lambda}= \mathcal{N}_{\lambda}^{+} \cup \mathcal{N}_{\lambda}^{-}$. Moreover, taking into account hypothesis $(\phi_{3})$, is possible to ensure that there exists an unique projection in each part $\mathcal{N}_{\lambda}^{-}, \mathcal{N}_{\lambda}^{+}$, see Section 2 ahead. In this way, we obtain that
problem \eqref{eq1} admits at least two positive solutions. These solutions are finding by standard minimization procedure in each part $\mathcal{N_{\lambda}}^{\pm}$. Thanks to hypothesis $(\phi_{3})$ is possible to guarantee that the fibering maps are in $C^{2}$ class which is essential in the Nehari method.

\nd We also have to deal with to the lack of compactness in $\w  \hookrightarrow   L^{\ell^{*}}(\Omega)$. In order to overcome the difficulty with compactness  we apply the concentration compactness principle, \cite{lions1,lions2,lions3,lions4}, together with variational methods as in \cite{BN-1}. In addition, the Brezis-Lieb Lemma for convex functions plays a crucial role.
\vskip.1cm

\nd It is worthwhile to mention that problem \eqref{eq1} admits at least two positive solutions thanks to the fact that the fibering maps give us an only projection in each of $\mathcal{N_{\lambda}}^{\pm}$, see Section 2 in the sequel. Those solutions are found by standard minimization procedure in each of $\mathcal{N_{\lambda}}^{\pm}$. The main tool here is to use hypothesis $(\phi_{3})$ showing that the fibering maps admits an unique critical point.

\nd The reader is also refered  to  \cite{G,AR,AGPeral,Fuk_1,GPeral,M,S,Willem} and references therein.

\nd The paper is organized as follows:  Section 2 is devoted to  proprieties of Nehari manifolds in our setting. In Section 3 we discuss on the fibering maps. Section 4 contains the proof of our main results. We use   $C, C_{1}, \ldots$ to denote positive constants.

\section{The Nehari manifold}

\nd The main goal in this section is to gather information on the critical points for the fibering maps associated to the energy functional $J_{\lambda}$. For an overview on the Nehari method we refer the reader to Willem \cite{Willem} and Brown et al \cite{Brown2,Brown1}.
\vskip.1cm

\nd The Nehari manifold associated to the functional $J_\lambda$ is
given by
\begin{equation}\begin{array}{rcl}\label{nehari}\mathcal{N}_\lambda&=&\{u\in W_0^{1,\Phi}(\Omega)\setminus \{0\}: \left<J'_{\lambda}(u),u\right>=0\}\\[2ex]
&=&\left\{u\in W_0^{1,\Phi}(\Omega)\setminus \{0\}:\Int\phi(|\nabla u|)|\nabla u|^2=\Int\lambda a(x)|u|^{q}+b(x)|u|^{\el}\right\}.
\end{array}\end{equation}
Later on, we shall prove that $ u \mapsto \langle J'_\lambda(u), u\rangle$ is  $C^{1}$  so that $\mathcal{N}_\lambda$ is a $\mathcal{C}^1$-submanifold of $W^{1,\Phi}_{0}(\Omega)$. Let $u\in \mathcal{N_\lambda}$. Using \eqref{nehari}, we infer that

\begin{equation}\label{eq2}
J_{\lambda}(u)=\Int \Phi(|\nabla u|) - \Fr{1}{q} \phi(|\nabla u|)|\nabla u|^2+\left(\Fr{1}{q}-\Fr{1}{\el}\right)b(x)|u|^{\el},
\end{equation}
or equivalently
\begin{equation}\label{eq3}
J_{\lambda}(u)=\Int \Phi(|\nabla u|) - \Fr{1}{\el} \phi(|\nabla u|)|\nabla u|^2-\lambda\left(\Fr{1}{q}-\Fr{1}{\el}\right)a(x)|u|^{q}.
\end{equation}

\nd As a first step we shall prove that $J_{\lambda}$ is coercive and bounded from below on $\mathcal{N}_\lambda$ which allows us to find a ground state that which gives us a critical point of $J_{\lambda}$. We have

\begin{prop}\label{coercive}
The functional $J_\lambda$ is coercive and bounded from below on $\mathcal{N}_\lambda$.
\end{prop}
\begin{proof}
In view of \eqref{conseqphi3} we get
\begin{equation}\label{eq4}
J_{\lambda}(u)\geq\left(\Fr{1}{m}-\Fr{1}{\el} \right)\Int\phi(|\nabla u|)|\nabla u|^2+\lambda\left(\Fr{1}{\el}-\Fr{1}{q}\right)\Int a(x)|u|^{q}.\nonumber
\end{equation}

\nd Now due the fact that
$$\min\{||u||^\ell,||u||^m\}\leq\Int \Phi(|\nabla u|)\leq \Fr{1}{\ell}\Int\phi(|\nabla u|)|\nabla u|^2,$$
 we  conclude that
\begin{equation}\label{functional1}\begin{array}{rcl}
J_{\lambda}(u)&\geq&\ell\left(\Fr{1}{m}-\Fr{1}{\el} \right)\min\{||u||^\ell,||u||^m\}+\lambda\left(\Fr{1}{\el}-\Fr{1}{q}\right)\Int a(x)|u|^{q}\\[3ex]
&\geq& \ell\left(\Fr{1}{m}-\Fr{1}{\el} \right)\min\{||u||^\ell,||u||^m\}+\lambda\left(\Fr{1}{\el}-\Fr{1}{q}\right)||a^+||_{\infty}\Int |u|^{q}.
\end{array}\end{equation}

\nd Since,  $W_0^{1,\Phi}(\Omega)\hookrightarrow L_{\Phi}(\Omega)\hookrightarrow L^\ell(\Omega)\hookrightarrow L^{q}(\Omega),$ there is  $C=C(q,\Phi)>0$ such that
$$J_{\lambda}(u)\geq \ell\left(\Fr{1}{m}-\Fr{1}{\el} \right)\min\{||u||^\ell,||u||^m\}+\lambda\left(\Fr{1}{\el}-\Fr{1}{q}\right)||a^+||_{\infty}C||u||^{q}.$$
Thus  $J_\lambda$ is coercive and bounded from from below on $\mathcal{N}_\lambda$. This ends the proof.
\end{proof}
\vskip.2cm

\nd At this moment we shall define the fibering map $\gamma_u: [0,+\infty)\to \mathbb{R}$  by
$$\gamma_u(t):=J_{\lambda}(tu)=\Int\Phi(t|\nabla u|)-\Fr{\lambda t^{q}}{q}a(x)|u|^{q}-\Fr{t^{\el}}{\el}b(x)|u|^{\el}.$$
Fibering maps have been considered together the Nehari manifold in order to ensure the existence of critical points for $J_{\lambda}$. In particular, for concave-convex nonlinearities it is important to know   the geometry for $\gamma_{u}$. Here we refer the reader to \cite{Brown2,Brown1,wu,wu3}.
\vskip.1cm

\nd Now we point out that $\gamma_{u}$ is of class $C^{1}$  thanks to  $(\phi_{1}) - (\phi_{2})$. More specifically, we obtain
\begin{equation*}\label{priderivada}
\gamma'_u(t)= \Int t\phi(t|\nabla u|)|\nabla u|^2-\lambda t^{q-1}a(x)|u|^{q}-t^{\el-1} b(x)|u|^{\el}.
\end{equation*}
It is easy to see that $tu\in\mathcal{N}_\lambda$ if and only if $\gamma'_u(t)=0.$ Therefore, $u\in\mathcal{N}_\lambda$ if and  only if $\gamma'_u(1)=0.$ In other words, it is sufficient to find stationary points of fibering maps in order to get critical points for $J_{\lambda}$ on $\mathcal{N}_\lambda$.
Notice also that, using  $(\phi_{3})$, we deduce that $\gamma_{u}$ is of class $C^{2}$  with second derivative given by
\begin{eqnarray}\label{segderivada}
	\gamma''_u(t)&= &\Int t\phi'(t|\nabla u|)|\nabla u|^3+\phi(t|\nabla u|)|\nabla u|^2-\lambda (q-1)t^{q-2}a(x)|u|^{q}dx\nonumber\\
	&-& (\el-1)t^{\el-2}\int_\Omega b(x)|u|^{\el}.\nonumber
\end{eqnarray}

As was pointed by Brown et al \cite{Brown2,Brown1} it is natural to divide $\N$ into three sets
$$\N^+:=\{u\in \N:\gamma''_u(1)>0\};$$
$$\N^-:=\{u\in \N:\gamma''_u(1)<0\};$$
$$\N^0:=\{u\in \N:\gamma''_u(1)=0\}.$$
Here we mention that $\N^+,~\N^-,~\N^0$ corresponds to critical points of minimum, maximum and inflection points, respectively. Here we refer the reader also to Tarantello \cite{tarantello}.

\begin{rmk}\label{gamma''}
It is not hard to verify that
\begin{equation}\label{eq5}
\begin{array}{rcl}
\gamma''_u(1)&=& \Int \phi'(|\nabla u|)|\nabla u|^3+(2-q)\phi(|\nabla u|)|\nabla u|^2-(\el-q)b(x)|u|^{\el}\\[2ex]
&=&
\Int \phi'(|\nabla u|)|\nabla u|^3+(2-\el)\phi(|\nabla u|)|\nabla u|^2-\lambda(q-\el)a(x)|u|^{q}.
\end{array}
\end{equation}
holds true for any $u\in \N$. Here was used identities \eqref{eq2} and \eqref{eq3}.
\end{rmk}

Now we shall prove that $\N$ is a $C^1$-manifold. This step is crucial in our argument in order to get the main result in this work.

\begin{lem}\label{c1} Suppose $(\phi_{1}) - (\phi_{3})$.  Then there exists  $\lambda_{1} > 0$  such that
\begin{enumerate}
\item $\N^0=\emptyset$.
\item $\N=\N^+\dot{\cup}\N^-$ is a $C^1$-manifold.
\end{enumerate}
\nd  for each $\lambda \in (0, \Lambda_{1})$.
\end{lem}

\proof First of all, we shall consider the proof for item (1). Arguing by contradiction we assume that $\N^0\neq\emptyset.$ Let $u\in\N^0$ be a fixed function. Clearly, we have $\gamma'_u(1)=\gamma''_u(1)=0.$ By \eqref{nehari} and \eqref{eq5}, we obtain, $$0=\gamma''_u(1)=\Int (2-q)\phi(|\nabla u| )|\nabla u|^2+\phi'(|\nabla u|)||\nabla u|^3+(q-\el)b|u|^{\el}.$$ Now taking into account \eqref{conseqphi3} we have that
$$(\ell-q)\Int \phi(|\nabla u| )|\nabla u|^2\leq(\el-q)||b^+||_{\infty}\|u\|_{\el}^{\el}\leq (\el-q)S_{\el}||b^+||_{\infty}||u||^{\el},$$ where $S_{\ell^*}$ is a best constant in the embedding $W_0^{1,\Phi}(\Omega)\hookrightarrow L^{\el}(\Omega).$ On the other hand,
$$(\ell-q)\Int \phi(|\nabla u| )|\nabla u|^2dx\geq\ell(\ell-q)\Int \Phi(|\nabla u|)dx\geq \ell(\ell-q)\min\{||u||^\ell,||u||^m\}. $$

Taking into account the estimates just above we observe that

$$\ell(\ell-q)\min\{||u||^\ell,||u||^m\}\leq(\el-q)S_{\el}||b^+||_{\infty}||u||^{\el}.$$
Therefore, we obtain
$$||u||^{\el}\geq \Fr{\ell(\ell-q)}{(\el-q)S_{\el}||b^+||_{\infty}}\min\{||u||^\ell,||u||^m\}=\left[\Fr{\ell(\ell-q)}{(\el-q)S_{\el}||b^+||_{\infty}}\right]||u||^{\alpha},$$
where we put $\alpha=\ell$ for any $\|u\|\geq 1$ and $\alpha=m$ for any $\|u\|\leq 1$. These facts imply that

\begin{equation}\label{des1.0}
	||u||\geq \left[\Fr{\ell(\ell-q)}{(\el-q)S_{\el}||b^+||_{\infty}}\right]^{\frac{1}{\el-\alpha}}.
\end{equation}

On the other hand, using \eqref{conseqphi3}, \eqref{eq5}, %$(\phi_3)$
and the Holder inequality (for Sobolev space), we obtain
$$(\el-m)\Int \phi(|\nabla u| )|\nabla u|^2\leq\lambda(\el-q)||a^+||_{(\frac{\ell}{q})'}\|u\|_{\ell}^{q}\leq\lambda(\el-q)S_{\ell}||a^+||_{(\frac{\ell}{q})'}||u||^{q},$$
where $S_{\ell}$ is a best constant in the embedding $W_0^{1,\Phi}(\Omega)\hookrightarrow L^{\ell}(\Omega).$ Using the same ideas discussed in the previous case we infer that
$$\ell(\el-m)\min\{||u||^\ell,||u||^m\}\leq \lambda(\el-q)S_{\ell}||a^+||_{(\frac{\ell}{q})'}||u||^{q}.$$
Hence, the last assertion says that
$$\Fr{\ell(\el-m)}{(\el-q)S_{\ell}||a^+||_{(\frac{\ell}{q})'}}||u||^\alpha=\Fr{\ell(\el-m)}{(\el-q)S_{\ell}||a^+||_{(\frac{\ell}{q})'}}\min\{||u||^\ell,||u||^m\}\leq \lambda||u||^{q}.$$
In this way, we mention that
\begin{equation}\label{des2}
\left[\Fr{\ell(\el-m)}{(\el-q)S_{\ell}||a^+||_{(\frac{\ell}{q})'}}\right]||u||^{\alpha-q}\leq \lambda.
\end{equation}
Under these conditions, using \eqref{des1.0} and \eqref{des2}, we get a contradiction for any
\begin{equation}\label{Lambda1}
\lambda> \left[\Fr{\ell(\ell-q)}{(\el-q)S_{\el}||b^+||_{\infty}}\right]^{\frac{\alpha-q}{\el-\alpha}}\left[\Fr{\ell(\el-m)}{(\el-q)S_{\ell}||a^+||_{(\frac{\ell}{q})'}}\right]=:
 \lambda_1.
\end{equation}
This finishes the proof of item $(1)$.
\vskip.1cm

Now we shall prove the item $(2)$. Without any loss of generality that we take $u\in \N^+$. Define $G(u):=\left<J'_\lambda(u),u\right>.$ It is no hard to see that
 $$G'(u)=J''_\lambda(u)\cdot(u,u)+\left<J'_\lambda(u),u\right>=\gamma''_u(1)>0,\,\,\ \forall u\in\N^+.$$
Hence, $0$  is a regular value for the functional $G$. Consequently, we see that $\N^+$ is a $C^1$-manifold. Similarly, we should be show that $\N^-$ is a $C^1$-manifold. As a consequence the proof of item $(2)$ follows due the fact that $\N^{0}=\emptyset$ for any $\lambda > 0$ small enough. This completes the proof. \hfill\cqd

\nd Now we are in a position  to prove that any critical point for $J_{\lambda}$ on $\mathcal{N}_\lambda$ is a free critical point, i.e, is a critical point in the whole space $W^{1,\Phi}_{0}(\Omega)$. More precisely, we shall consider the following result

\begin{lem}\label{criticalpoint}
Suppose $(\phi_{1}) - (\phi_{3})$. Let $u_0$  be a local minimum (or local maximum) for $J_\lambda$ on $\mathcal{N}_{\lambda}$. Then
$u_0$ is a critical point of $J_\lambda$ on $\w$ for each $\lambda<\lambda_1$.
\end{lem}
\proof Let $u_{0} \in \N$ be a local maximum or mininum for the functional $J_{\lambda}$ on $\N$. Without any loss of generality we assume that $u_0$ is a local minimum. Define the function
$$\theta(u)=\left<J_\lambda'(u),u\right>=\Int\phi(|\nabla u|)|\nabla u|^2-\lambda a(x) |u|^{q}-b(x)|u|^{\el}.$$
It is easy to see that $u_0$ is a solution for the minimization problem
\begin{equation}\label{ll}
\left\{\begin{array}{rcl}
\min J_\lambda(u),\\
\theta(u)=0
\end{array}\right.
\end{equation}
Arguing as in Carvalho et al \cite{JME}, we infer that
$$\left<\theta'(u),v\right>=\Int \phi'(|\nabla u|)|\nabla u|^2\nabla v+2\phi(|\nabla u|)\nabla u\nabla v-\lambda qa(x)|u|^{q-1}v-\el b(x)|u|^{\el-1}v$$ holds true for any $u, v  \in W^{1,\Phi}_{0}(\Omega)$. As a consequence, taking $u = v = u_{0}$, we observe that $$
\begin{array}{rcl}\left<\theta'(u_0),u_0\right>&=&\Int \phi'(|\nabla u_0|)|\nabla u_0|^3+2\phi(|\nabla u_0|)|\nabla u_0|^2\\[2ex]
&-&\Int\lambda qa(x){|u_0|}^{q}-\el b(x){|u_0|}^{\el}.\end{array}$$
Moreover, using the fact that $u_0\in \N^{+}$, \eqref{nehari} and \eqref{eq5}, we deduce that
$$\begin{array}{rcl}
\left<\theta'(u_0),u_0\right>&=&\Int \phi'(|\nabla u_0|)|\nabla u_0|^3+(2-q)\phi(|\nabla u_0|)|\nabla u_0|^2\\[2ex]
&-&(\el-q)\Int b(x){|u_0|}^{\el}=\gamma''_u(1)>0.\end{array}$$

In view of Lemma \ref{c1} we mention that problem \eqref{ll} admits at least one solution in the following form
$$J_\lambda'(u_0)=\mu\theta'(u_0)$$
where $\mu\in\R$ is given by the Lagrange Multipliers Theorem. As a consequence
\begin{equation*}
\mu\left<\theta'(u_0),u_0\right> \,=\, \left<J_\lambda'(u_0),u_0\right> \,=\,0.
\end{equation*}
 Furthermore, we know that  $\left<\theta'(u_0),u_0\right>\neq 0.$ This assertion implies that $\mu=0$, i.e, $u_{0}$ is a critical point for $J_{\lambda}$ on $W^{1,\Phi}_{0}(\Omega)$. The proof for this lemma is now complete.\hfill\cqd

\section{Analysis of the Fibering Maps}

In this section we give a complete description on the geometry for the fibering maps associated
to the problem \eqref{eq1}. Let $u \in \w\backslash\{0\}$ be a fixed function.  To the best our knowledge the essential nature of fibering maps is
determined by the signs of $\Int a(x)|u|^q$ and $\Int b(x)|u|^{\el}$.
Throughout this section is useful to consider the auxiliary function of $C^{1}$ class given by
$$m_u(t)=t^{2-q}\Int \phi(t|\nabla u|)|\nabla u|^2-t^{\el-q}\Int b(x)|u|^{\el}, t \geq 0, u \in \w.$$

Now we shall consider a result comparing points $t u \in \mathcal{N}_\lambda$ with the the function $m_{u}$. More precisely, we have

\begin{lem}\label{m_uegamma_u}
Let $t>0$ be fixed. Then $tu\in \N$ if and only if $t$ is a solution of
\begin{equation*}
m_u(t)=\lambda\Int a(x) |u|^{q}.
\end{equation*}

\proof Fix $t>0$ in such may that $tu\in\N$. Then
$$t^2\Int\phi(\nabla(tu))|\nabla u|^2=t^{q}\lambda\Int a(x)|u|^{q}+t^{\el}\Int b(x)|u|^{\el}.$$
The identity just above is equivalent to
$$t^2\Int\phi(|\nabla t u |)|\nabla u|^2- t^{\el}\Int b(x)|u|^{\el}= t^{q}\lambda\Int a(x)|u|^{q}.$$
Multiplying the above expression by $t^{-q},$ we get
$$t^{2-q}\Int\phi(|\nabla t u|)|\nabla u|^2 -t^{\el-q}\Int b(x)|u|^{\el}= \lambda\Int a(x)|u|^{q}.$$
In view the definition of $m_u$ we obtain the desired result. This ends the proof. \hfill\cqd
\end{lem}

The next lemma is a powerful tool in order to get a precise information around the function $m_{u}$ and the fibering maps. More precisely, we shall consider the following result

\begin{lem}\label{m_u-comp}
\begin{enumerate} \item Suppose that $\Int b(x)|u|^{\el}\leq 0$ holds. Then we obtain $\displaystyle m_{u}(0) := \lim_{t \rightarrow 0} m_{u}(t)=0 , m_{u}(\infty) :=  \lim_{t \rightarrow \infty} m_{u}(t) = \infty$ and $m'_u(t)>0$ for any $t>0$.
\item Suppose $\Int b(x)|u|^{\el}>0$ and $(H)$. Then
 there exists an only critical point for $m_{u}$, i.e, there is an only point $\tilde{t}>0$ in such way that $m'_u(\tilde{t})=0$.
 Furthermore, we know that $\tilde{t} > 0$ is a global maximum point for $m_{u}$ and $m_{u}(\infty) = - \infty$.
\end{enumerate}
\proof
 Initially we observe that
 \begin{eqnarray}
	 m'_u(t)& = & (2-q)t^{1-q}\Int \phi(t|\nabla u|)|\nabla u|^2+t^{2-q}\Int \phi'(|\nabla (tu)|)|\nabla u|^3\nonumber\\
	  & - &(\el-q)t^{\el-q-1}\Int b(x)|u|^{\el}.\nonumber
 \end{eqnarray}

\nd Now we shall prove the item (1). Additionally, taking into account Remark \ref{conseqphi3} it is easy to verity that
 \begin{equation}\label{ee2}
 \ell-2\leq\Fr{\phi'(t)t}{\phi(t)}\leq m-2, \,\,\mbox{for any} \,\,\,t \geq 0. \end{equation}
 As a consequence we see that
$$\begin{array}{rcl}
m'_u(t)&\geq& (2-q)t^{1-q}\Int \phi(t|\nabla u|)|\nabla u|^2\\[2ex]&+&\Int (\ell-2)t^{1-q}\phi(|\nabla tu|)|\nabla u|^2-(\el-q)t^{\el-q-1}b(x)|u|^{\el}\\[2ex]
&=&\Int (\ell-q)t^{1-q} \phi(t|\nabla u|)|\nabla u|^2-(\el-q)t^{\el-q-1}b(x)|u|^{\el}>0.
\end{array}$$ Hence the function $m_u$ is increasing for any $t>0,$ i.e, we have $m^{'}_{u}(t) > 0$ for any $t > 0$.
Moreover, we shall prove that $m_{u}(0) = 0$. In fact, using \cite[Lemma 2.1]{Fuk_1}, we deduce that
\begin{equation}\label{ae}
\Int t^{m-q} \phi(|\nabla u|)|\nabla u|^2-t^{\el-q}b(x)|u|^{\el}\leq m_u(t),
\end{equation}
and
\begin{equation}\label{aee}
m_u(t)\leq \Int t^{\ell-q} \phi(|\nabla u|)|\nabla u|^2-t^{\el-q}b(x)|u|^{\el}, t \in [0, 1].
\end{equation}
Taking the limits in estimates \eqref{ae} and \eqref{aee} we get $\displaystyle \lim_{t \rightarrow 0^+} m_{u}(t) = 0$.
Furthermore, arguing as in the proof \eqref{ae}, we obtain
\begin{equation}
m_{u}(t) \geq t^{\ell-q}  \int_{\Omega} \phi(|\nabla u|)|\nabla u|^2 - t^{\el-q} \Int b(x)|u|^{\el}, t \geq 1.\nonumber
\end{equation}
Due the fact that $\ell > q $ the last assertion implies that $m_{u}(\infty) = \displaystyle \lim_{t \rightarrow \infty} m_{u}(t) = \infty$.
This finishes the proof of item (1).

Now we shall prove the item $(2)$.  As first step we mention that $m_u$ is increasing for $ t \in (0,1)$ and $\ds\lim_{t\to\infty}m_u(t)=-\infty$. More specifically, using one more time \eqref{ee2} we get

$$\begin{array}{rcl}
m'_u(t)&\geq&
\Int (\ell-q)t^{1-q}\phi(t|\nabla u|)|\nabla u|^2-(\el-q)t^{\el-q-1}b(x)|u|^{\el}\\[2ex]
&\geq & \Int(\ell-q)t^{1-q}t^{m-2} \phi(|\nabla u|)|\nabla u|^2-(\el-q)t^{\el-q-1}b(x)|u|^{\el}\\[2ex]
&=& \dfrac{1}{t}\Int(\ell-q)t^{m-q} \phi(|\nabla u|)|\nabla u|^2-(\el-q)t^{\el-q}b(x)|u|^{\el}.
\end{array}$$
Since $m<\el$ we mention that $m'_u(t)>0$ for any $t \in (0, 1)$. Furthermore, arguing as above we see also that
$$m_u(t)\leq \Int t^{m-q}\phi(|\nabla u|)|\nabla u|^2-t^{\el-q}b(x)|u|^{\el}, t \geq 1.$$
Therefore, we deduce that $\ds\lim_{t\to\infty}m_u(t)=-\infty$ where was used the fact that $m<\el$.

Now the main goal in this proof is to show that $m_u$ has an unique critical point $\tilde{t}>0.$
Note that, we have $m^{\prime}_u(t)=0$ if and only if $$(2-q)t^{2-\el}\Int \phi(t|\nabla u|)|\nabla u|^2+t^{3-\el}\Int \phi'(|\nabla (tu)|)|\nabla u|^3=(\el-q)\Int b(x)|u|^{\el}.$$
Define the auxiliary function $\eta_{u}: (0,\infty) \rightarrow \mathbb{R}$ given by
$$\eta_u(t)=(2-q)t^{2-\el}\Int \phi(t|\nabla u|)|\nabla u|^2+t^{3-\el}\Int \phi'(|\nabla (tu)|)|\nabla u|^3.$$
 Here we emphasize that
\begin{equation}\label{ee3}
\ds\lim_{t\to 0^+}\eta_u(t)=+\infty.
\end{equation}
Indeed, arguing as in previous cases and putting $0<t<1$, we easily see that
 $$\begin{array}{rcl}\eta_u(t)&\geq& (2-q)t^{2-\el}\Int \phi(t|\nabla u|)|\nabla u|^2+t^{2-\el}(\ell-2)\Int \phi(|\nabla (tu)|)|\nabla u|^2\\[2ex]
 &=&(\ell-q)t^{2-\el} \Int \phi(|\nabla (tu)|)|\nabla u|^2\\[2ex]
 &\geq& (\ell-q)t^{2-\el}t^{m-2} \Int \phi(|\nabla u|)|\nabla u|^2\\[2ex]
 &=& (\ell-q)t^{m-\el} \Int \phi(|\nabla u|)|\nabla u|^2.
 \end{array}$$
 Using one more time that $m < \el$ and $\ell > q $ it follows that \eqref{ee3} holds true.

 On the other hand, we mention that $\eta_{u}$ is a decreasing function which satisfies
\begin{equation}\label{ee33}
\lim_{t\to \infty}\eta_u(t)=0.
\end{equation}
In fact, taking into account \cite[Lemma 2.1]{Fuk_1}, for any $t>1$, we observe that

\begin{equation}\label{ee4}
\eta_u(t)\leq (m-q)t^{m-\el}\Int \phi(|\nabla u|)|\nabla u|^2
\end{equation}
and
\begin{equation}\label{ee5}
\eta_u(t)\geq (\ell-q)t^{\ell-\el}\Int \phi(|\nabla (u)|)|\nabla u|^2.
\end{equation}
Hence \eqref{ee4} and \eqref{ee5} say that \eqref{ee33} holds true. Moreover, we have also that
$$\begin{array}{rcl}
\eta_u'(t)&=&\Int[(2-\el)(2-q)t^{1-\el}\phi(t|\nabla u|)|\nabla u|^2 \\[2ex]
&+&(5-(\el+q))t^{2-\el}\phi'(t|\nabla u|)|\nabla u|^3]+t^{3-\el}\Int \phi''(t|\nabla u|)|\nabla u|^4.
\end{array}$$
Using hypothesis $(\phi_3)$ and  Remark \ref{conseqphi3} we mention that
 $$\left\{\begin{array}{rcl}
 t^2\phi''(t)&\leq& (m-4)t\phi'(t)+(m-2)\phi(t),\\
 t^2\phi''(t)&\geq& (\ell-4)t\phi'(t)+(\ell-2)\phi(t).
 \end{array}\right.$$
 As a consequence the estimates just above imply that
 $$\begin{array}{rcl}\eta_u'(t)&\leq&\Int(2-\el)(2-q)t^{1-\el}\phi(t|\nabla u|)|\nabla u|^2\\[2ex]
&+&\Int(5-(\el-q))t^{2-\el}\phi'(t|\nabla u|)|\nabla u|^3 \\ [2ex]
&+&\Int [(m-4)t^{2-\el}\phi'(t|\nabla u|)|\nabla u|^3+(m-2)t^{1-\el}\phi(t|\nabla u|)|\nabla u|^2]\\[2ex]
&=& \Int[((2-\el)(2-q)+m-2)t^{1-\el}\phi(t|\nabla u|)|\nabla u|^2]+\\[2ex]
&+&\Int [(m+1)-(\el+q))t^{2-\el}\phi'(t|\nabla u|)|\nabla u|^3]\end{array}$$
%&\leq&
Note that the first part of hypothesis $(H)$ implies that
\begin{equation*}
(\el-1)(m-\ell)<(\el-\ell)(m-q).
\end{equation*}
Moreover, we mention that $((2-\el)(2-q)+m-2)+((m+1)-(\el+q))(\ell-2) = (\el-1)(m-\ell)-(\el-\ell)(m-q)$.
Under these conditions it is no hard to verify that
$$
\begin{array}{rcl}
\eta_u'(t)&\leq& \Int[((2-\el)(2-q)+m-2)t^{1-\el}\phi(t|\nabla u|)|\nabla u|^2+\\[2ex]
&+&\Int ((m+1)-(\el+q))(\ell-2)t^{1-\el}\phi(t|\nabla u|)|\nabla u|^2]\\
&=&((2-\el)(2-q)+m-2)t^{1-\el}\Int\phi(t|\nabla u|)|\nabla u|^2\\[2ex]
&+&((m+1)-(\el+q))(\ell-2)]t^{1-\el}\Int\phi(t|\nabla u|)|\nabla u|^2\\[2ex]
&=&[(\el-1)(m-\ell) - (\el-\ell)(m-q)]t^{1-\el}\Int\phi(t|\nabla u|)|\nabla u|^2
<0.
\end{array}$$
Thus we conclude that $\eta_u$ is decreasing function proving that $m_u $ has an unique critical point which is a maximum critical point for $m_{u}$. The proof for this lemma is now complete.
\hfill\cqd
\end{lem}

Now we shall prove that $m_{u}$ has a behavior at infinity and at the origin given by the sings of $\int_{\Omega} a(x)|u|^{q}$ and
$\int_{\Omega} b(x)|u|^{\el}$. This is crucial in to prove a complete description on the geometry for the fibering maps.

\begin{lem}\label{fib}
	Let $u\in W^{1,\Phi}_0(\Omega)/\{0\}$ be a fixed function. Then we shall consider the following assertions:
\begin{enumerate}
		\item Assume that $\Int b(x)|u|^{\el} \leq 0$. Then $\gamma'_u(t)\neq 0$ for any $t>0$ and $\lambda > 0$ whenever $\Int a(x)|u|^{q} \leq 0$. Furthermore, there exist an unique $ t_1= t_1(u,\lambda)$ in such way that $\gamma'_u(t_1)=0$  and $ t_1 u\in \N^+$ whenever $\Int a(x)|u|^{q} > 0.$
	
\item Assume that $\Int b(x)|u|^{\el} > 0$ holds. Then there exists an unique $t_1=t_1(u,\lambda)> \tilde t$ such that
$\gamma'_u(t_1)= 0$ and $t_1 u\in \N^-$ whenever $\Int a(x)|u|^{q} \leq 0$.
\item Assume that $(H)$ holds. For each $\lambda>0$ small enough there exists unique $0<t_1=t_1(u,\lambda)<\tilde t<t_2=t_2(u,\lambda)$ such that 	$\gamma'_u(t_1)=\gamma'_u(t_2)=0$, $ t_1 u\in \N^+$ and $ t_2 u\in \N^-$ whenever $\Int a(x)|u|^{q}> 0$, $\Int b(x)|u|^{\el}> 0$ holds.
\end{enumerate}	
\end{lem}
\proof
First of all, we shall consider the proof for the case $\Int b(x)|u|^{\el}\leq 0$ and $\Int a(x)|u|^{q}\leq 0$. Using Lemma \ref{m_u-comp} (1) it is easy to verify that
\begin{equation}\label{m_u}
	m_u(0)=0,~\lim_{t\rightarrow\infty}m_u(t) = \infty \,\,\mbox{and} \,\, m'_u(t)>0, t \geq 0.\nonumber
\end{equation}
Under these conditions we deduce that
$$
m_u(t)\neq \lambda\Int a(x)|u|^{q} \,\, \mbox{for any} \,\,t>0, \lambda>0.
$$
According to Lemma \ref{m_uegamma_u} we deduce that $tu\not\in \N$ for any  $t>0$. In particular, we see also that $\gamma'(t)\neq 0$ for each $t>0$.

Now we shall consider the proof for the case $\Int a(x)|u|^{q} > 0$ and $\Int b(x)|u|^{\el} \leq 0$. Using one more time Lemma \ref{m_u-comp} (1) we observe that $m_{u}(0)= 0 , m_{u}(\infty) = \infty$ and $m_{u}$ is a increasing function. In particular, the equation
$$m_u(t)= \lambda \Int a(x)|u|^{q}$$
admits exactly one solution $t_1= t_1(u,\lambda)>0$. Hence, using Lemma \ref{m_uegamma_u}, we know that $t_1 u\in \N$ proving that $\gamma'_u( t_1) =0$. Additionally, using the identity
\begin{equation}\label{rel-m-u-gamma-u}
	m_u(t)=t^{1-q}\gamma_u'(t)+\lambda \Int a(x)|u|^{q},\nonumber
\end{equation}
we easily see that
$$0<m'_u(t_1) = t_1^{1-q}\gamma''_u(t_1).$$
In particular, we have been proven that $t_1u\in\N^+$.

Now we shall consider the proof for the case $\Int a(x) |u|^{q } \leq 0$ and $\Int b(x) |u|^{\el }  > 0$. Here the function $m_{u}$
admits an unique turning point $\tilde{t} > 0$, i.e, we have that  $m^{\prime}_{u}(t) = 0, t > 0$ if only if $t = \tilde{t}$, see Lemma \ref{m_u-comp} (2).
Moreover, $\tilde{t}$ is a global maximum point for $m_{u}$ in such way that $m_{u}(\tilde{t}) > 0, m_{u}(\infty) = - \infty$.
As a product there exits an unique $t_1>\tilde t$ such that
$$m_u(t_1)= \lambda\Int a(x)|u|^{q}.$$
Here we emphasize that $m_u'(t_1)<0$ where we have used the fact that $m_{u}$ is a decreasing function in $(\tilde{t}, \infty)$. As a consequence we obtain $0>m'_u(t_1)=t_1^{1-q}\gamma_u''(t_1)$ proving that $t_1u\in\N^-$.

At this moment we shall consider the proof for the case $\Int a(x)|u|^{q} > 0$ and $\Int b(x)|u|^{\el} > 0$.
Due the fact that $\Int a(x)|u|^{q}>0$ we obtain  $\bar{\lambda}_1>0$  such that
\begin{equation}\label{lambda1barra}
m_u(\tilde t) > \lambda\Int a(x)|u|^{q}, \,\,\mbox{for any} \,\, \lambda\in (0,\bar{\lambda}_1).
\end{equation}
It is worthwhile to mention that $m_{u}$ is increasing in $(0, \tilde{t})$ and decreasing in $(\tilde{t}, \infty)$.
It is not hard to verify that there exist exactly two points $0<t_1=t_1(u,\lambda)  <\tilde t < t_2=t_2(u,\lambda)$ such that
$$m_u(t_i)=\lambda\Int a(x)|u|^{q},~i=1,2.$$
Additionally, we have that $m_u'(t_1)>0$ and $m_u'(t_2)<0$. Arguing as in the previous step we ensure that $t_1u\in\N^+$ and $t_2u\in\N^-$. This completes the proof.  \hfill\cqd
\vskip.2cm

The next lemma shows that for any $\lambda>0$ small enough  the function $\gamma_u$ assumes  positive values. This is crucial for the proof of our
main theorems proving that $\gamma_u$ admits one or two critical points. At the same time, we shall show also that $J_\lambda$ is away form zero on the Nehari manifold $\N^{-}$. In particular, any critical point for $J_{\lambda}$ on
$\N^{-}$ provide us a nontrivial critical point.

\vskip.2cm

\begin{lem}\label{nehari-} There exist $\delta_1, \tilde{\lambda}_1>0$
 in such way that $J_\lambda(u)\geq\delta_1$ for any $u\in \N^{-}$
where $0<\lambda < \tilde{\lambda}_1$.
\end{lem}
 \proof Since $u\in\N^-(\Omega),$ we have that
$\gamma''_{u}(1)<0$. Arguing as in the proof of Lemma \ref{c1}, we obtain $$||u||>\left[\Fr{\ell(\ell-q)}{(\el-q)S_{\el}\|b^+\|^{\infty}}\right]^{\frac{1}{\el-\alpha}}.$$
Moreover, in view of \eqref{functional1} and the Sobolev imbedding, we have that
$$\begin{array}{rcl}J_{\lambda}(u)&\geq&\ell\left(\Fr{1}{m}-\Fr{1}{\el} \right)\min\{||u||^\ell,||u||^m\}+\lambda\left(\Fr{1}{\el}-\Fr{1}{q}\right)\Int a(x)|u|^{q}\\[2ex]
&=&\ell\left(\Fr{1}{m}-\Fr{1}{\el} \right)||u||^\alpha+\lambda\left(\Fr{1}{\el}-\Fr{1}{q}\right)\Int a(x)|u|^{q}\\[2ex]
&\geq&
\ell\left(\Fr{1}{m}-\Fr{1}{\el} \right)||u||^\alpha+\lambda\left(\Fr{1}{\el}-\Fr{1}{q}\right)\|a^+\|_{\left(\frac{\ell}{q}\right)'}S_{\ell}||u||^q\\[2ex]
&=&||u||^q\left[\ell\left(\Fr{1}{m}-\Fr{1}{\el} \right)||u||^{\alpha-q}+\lambda\left(\Fr{1}{\el}-\Fr{1}{q}\right)\|a^+\|_{\left(\frac{\ell}{q}\right)'}S_{\ell}\right].
\end{array}$$
Using the inequalities just above we get
$$J_{\lambda}(u)>\left[\Fr{\ell(\ell-q)}{(\el-q)S_{\el}\|b^+\|_{\infty}}\right]^{\frac{q}{\el-\alpha}}\left[A +\lambda B \right]$$
where
$$A = \ell\left(\Fr{1}{m}-\Fr{1}{\el} \right)\left(\Fr{\ell(\ell-q)}{(\el-q)S_{\el}\|b^+\|_{\infty}}\right)^{\frac{\alpha-q}{\el-\alpha}}$$
and
$$B = \left(\Fr{1}{\el}\\-\Fr{1}{q}\right)\|a^+\|_{\left(\frac{\ell}{q}\right)'}S_{\ell}.$$
Therefore, for each $0<\lambda<\tilde{\lambda}_1:=\Fr{q}{m}\lambda_1$%\Lambda_1$
 where we take $\lambda_{1}>0 $ given by \eqref{Lambda1}. Here we put $\tilde{\lambda}_1:=\Fr{q}{m}\lambda_1$ obtaining the desired result. This finishes the proof. \hfill\cqd

Now we shall prove that any minimizer on $\N^{+}$ has negative energy. More specifically, defining
$\alpha_\lambda:=\displaystyle\inf_{u\in \N}J_\lambda(u) ,  \alpha_\lambda^+=\displaystyle\inf_{u\in \N^+} J_\lambda(u)$
we can be shown the following result

\begin{lem}\label{nehari+} Suppose $(H)$. Then there exist $u \in \N^{+}$ and $\lambda_{1} > 0$ in such way that $\alpha_{\lambda}^{+} \leq J_\lambda(u) < 0$ for each $0<\lambda < \lambda_1$. In particular, we obtain $\alpha_{\lambda} = \alpha_{\lambda}^{+}$ for each $0<\lambda < \lambda_1$.
\end{lem}

\proof Fix $u\in \N^+$. Here we observe that $\gamma_u''(1)>0$. As a consequence
$$
\begin{array}{rcl}
\displaystyle(\el-q)\int_\Omega b(x)|u|^{\el}&<&\Int\phi'(|\nabla u|)|\nabla u|^3+(2-q)\phi(|\nabla u|)|\nabla u|^2
\\[2ex]
&\leq&\Int(m-2)\phi(|\nabla u|)|\nabla u|^2+(2-q)\phi(|\nabla u|)|\nabla u|^2\\
&=&(m-q)\Int\phi(|\nabla u|)|\nabla u|^2. \end{array}$$
The last inequalities imply that
\begin{equation*}
\int_\Omega b(x)|u|^{\el}<\Fr{m-q}{\el-q}\Int\phi(|\nabla u|)|\nabla u|^2.
\end{equation*}
On the other hand, using the inequality just above and \eqref{conseqphi3} we see that
 we easily see that
$$\begin{array}{rcl}
J_\lambda(u)&\leq& \left(\Fr{1}{\ell}-\Fr{1}{q}\right)\Int\phi(|\nabla u|)|\nabla u|^2+\left(\Fr{1}{q}-\Fr{1}{\el}\right)b(x)|u|^{\el}\\[3ex]
&<&\left[\left(\Fr{1}{\ell}-\Fr{1}{q}\right)+\left(\Fr{1}{q}-\Fr{1}{\el}\right)\left(\Fr{m-q}{\el-q}\right)\right]\Int\phi(|\nabla u|)|\nabla u|^2\\[3ex]
&=&\Fr{1}{q}\left[\Fr{q-\ell}{\ell }+\Fr{m-q}{\el }\right]\Int\phi(|\nabla u|)|\nabla u|^2.
\end{array}$$
In view of hypothesis $(H)$ it follows that $\alpha^+_\lambda \leq J_{\lambda}(u) <0$. Additionally, we stress that $\N=\N^-\cup\N^+$ and $\alpha_\lambda^->0$. Hence we deduce that $\alpha_\lambda^+=\alpha_\lambda$. This completes the proof. \hfill\cqd

\section{The Palais-Smale condition}

In this section we shall prove some auxiliary results in order to get the Palais-Smale condition for the functional $J_{\lambda}$ on the Nehari manifold.
In general, given any Banach space $X$ space endowed with the norm $\|\|$ and taking $I : X \rightarrow \mathbb{R}$ a functional of $C^{1}$ class we recall that a sequence $(u_{n}) \in X$ is said to be a Palais-Smale sequence at level $c \in \mathbb{R}$, in short $(PS)_{c}$, when $I(u_{n}) \rightarrow c$ and $I^{\prime}(u_{n}) \rightarrow 0$ as $n \rightarrow \infty$. Recall that $I$ satisfies the Palais-Smale condition at the level $c$, in short $(PS)_{c}$ condition, when any $(PS_{c})$ sequence admits a convergent subsequence. We say simply that $I$ verifies the Palais-Smale condition when $(PS)_{c}$ condition holds true for any $c \in \mathbb{R}$.

Here we follow same ideas discussed in Tarantello \cite{tarantello}.

\begin{lem}\label{lem1ps} Suppose $(\phi_{1})- (\phi_{3})$ and $(H)$. Let $u \in \mathcal{N}^{+}$ be fixed. Then there exist $\epsilon>0$ and a differentiable function $$\xi:B(0,\epsilon)\subset W_{0}^{1,\Phi}(\Omega)\to (0, \infty),\ \ \ \xi(0)=1,\,\ \xi(v)(u-v)\in \mathcal{N}^{+}, v  \in B(0,\epsilon).$$
Furthermore, we have that
\begin{eqnarray}\label{ps}
\left<\xi'(0),v\right> &=& \dfrac{1}{\gamma''_u(1)} \Int\left\{[\phi'(|\nabla u|)|\nabla u|+2\phi(|\nabla u|)]\nabla u\nabla v-\el b(x)|u|^{\el-2}uv\right.\nonumber\\
&-&\left.q\lambda a(x)|u|^{q-2}uv\right\}.
\end{eqnarray}
\end{lem}
\proof Initially, we define $\psi: \w \backslash \{ 0\} \rightarrow \mathbb{R}$ given by $\psi(u)=\left<J'_\lambda(u),u\right>$ for $u \in \w \backslash \{ 0\}$. It is easy to verity that
$$\left<\psi'(u),u\right>=\Int \phi'(|\nabla u|)|\nabla u|^3+2\phi(|\nabla u|)|\nabla u|^2-\el b(x)|u|^{\el}-q\lambda a(x)|u|^{q}.$$
Recall that $\left<\psi'(u),u\right>=\gamma''_u(1)$ holds for any $u\in\N$ where $\gamma''_u(1)$ is given by Remark \ref{gamma''}.

Now we define $F_u : \mathbb{R} \times \w\backslash \{0\} \rightarrow \mathbb{R}$ given by $$F_{u}(\xi,w)= \left<J'_\lambda(\xi(u-w)),\xi(u-w)\right>.$$ Here we observe that $F_u(1,0)=\psi(u)$. As a consequence
\begin{eqnarray}
 \Fr{d}{d\xi}F_u(\xi,w)&=& 2\xi \Int \phi(\xi|\nabla(u-w)|)|\nabla(u-w)|^2+\xi^2\phi'(\xi|\nabla(u-w)|)|\nabla(u-w)|^3 \nonumber \\
 &-& \el\xi^{\el-1}\Int b(x)|u-w|^{\el}-q\xi^{q-1}\lambda a(x)|u-w|^q.\nonumber
\end{eqnarray}
In particular, for each $u\in\N$, we mention that
 \begin{eqnarray}
 \Fr{d}{d\xi}F_u(1,0)&=&\Int 2\phi(|\nabla u|)|\nabla u|^2+\phi'(|\nabla u|)|\nabla u|^3 \nonumber \\
 &-& \el\Int b(x)|u|^{\el}-q\lambda a(x)|u|^q=\gamma_u''(1)\neq 0. \nonumber
 \end{eqnarray}
 As a product, using the Inverse Function Theorem, there exist $\epsilon>0$ and a differentiable function $\xi:B(0,\epsilon)\subset W^{1,\Phi}(\Omega)\to (0, \infty)$ satisfying $\xi(0)=1$ and $F_u(\xi(w),w)= \left<J'_\lambda(\xi(u-v),\xi(u-v))\right> = 0,$ i.e. $\xi(w)(u-w)\in\N,\ \ \forall w\in B(0,\epsilon).$ Furthermore, we also obtain
 \begin{equation*}
 \left<\xi'(w),v\right>=-\Fr{\left<\partial_2F_u(\xi(w),w),v\right>}{\partial_1F_u(\xi(w),w)}, \left<\xi'(0),v\right>=-\Fr{\left<\partial_2F_u(\xi(0),0),v)\right>}{\partial_1F_u(\xi(0),0)}.
 \end{equation*}
 Here $\partial_1F_u$ and $\partial_2F_u$ denote the partial derivatives on the first and second variable, respectively.

 On the other hand, after some manipulations we see that
 $$\begin{array}{rcl}\left<\partial_2F_u(\xi(w),w),v\right>&=&\xi^2\Int\phi'(\xi|\nabla (u-w)|\Fr{\left<\nabla(u-w),-\nabla v\right>}{|\nabla(u-w)|}|\nabla(u-w)|^2)\\[2ex]
 &+&
 2\xi^2\Int\phi(\xi|\nabla (u-w))\left<\nabla(u-w),-\nabla v\right>\\[2ex]&-&\el\xi^{\el}\Int b(x)|u-w|^{\el-2}(u-w)(-v)\\[2ex]
 &-&\lambda q \xi^{q}\Int a(x)|u-w|^{q-2}(u-w)(-v).
 \end{array}$$
 Hence, putting $w=0$ and $\xi=\xi(0)=1$, the last identity just above shows that
 $$\begin{array}{rcl}-\left<\partial_2F_u(1,0),v\right>&=&\Int\phi'(|\nabla u|){\left<\nabla u,\nabla v\right>}{|\nabla u|}
 +
 2\Int\phi(|\nabla u|)\left<\nabla u,\nabla v\right>\\[2ex]&-&\el\xi^{\el}\Int b(x)|u|^{\el-2} uv
 -\lambda q a(x)|u|^{q-2}uv
 \end{array}$$
 Here was used the fact that $\partial_1 F_u(1,0)=\gamma_{u}''(1)$ holds for any $u\in\N$. The proof is now finished. \hfill\cqd

Analogously, using the same ideas discussed in the proof previous result, we get the following result

\begin{lem}\label{lem2ps}
Suppose $(\phi_{1})- (\phi_{3})$ and $(H)$. Let $u \in \N^{-}$ be fixed. Then there are $\epsilon>0$ and a differentiable function $$\xi^-:B(0,\epsilon)\subset W^{1,\Phi}(\Omega)\to (0, \infty),\ \ \ \xi^-(0)=1,\,\ \xi^-(v)(u-v)\in\N^{-}, \, v \in B(0,\epsilon).$$
Furthermore, we obtain
\begin{eqnarray}\label{ps-}
\left< (\xi^{-})^{\prime}(0),v\right>&=&\dfrac{1}{\gamma''_u(1)} \Int\left\{[\phi'(|\nabla u|)|\nabla u|+2\phi(|\nabla u|)]\nabla u\nabla v-\el b(x)|u|^{\el-2}uv\right.\nonumber\\
&-&\left.q\lambda a(x)|u|^{q-2}uv\right\}.
\end{eqnarray}
\end{lem}

 In the next result we shall prove that any minimizer sequence for the functional $J$ in $\mathcal{N}_{\lambda}^{-}$ or $\mathcal{N}_{\lambda}^{+}$ is bounded from below and above for some positive constants. This is crucial in order to get a minimizer on the Nehari manifold.

\begin{prop} Suppose $(\phi_{1})- (\phi_{3})$ and $(H)$.
	Let $(u_n)$ be a minimizer sequence for the functional $J$ on the Nehari manifold $\mathcal{N}_{\lambda}^{+}$. Then
	\begin{equation}\label{12b}
	\liminf_{n \rightarrow \infty} ||u_n|| \geq - \alpha_\lambda^{\frac{1}{q}} \left[\Fr{\el q}{(\el-q)\lambda\|a\|_\infty S^{q}_{q}}\right]^{\frac{1}{q}} > 0
	\end{equation}
	and
	\begin{equation}\label{u_n-lim}
	 ||u_n||< \left[\Fr{\lambda}{q}\left(\Fr{\el-q}{\el-m}\right)\|a\|_\infty S^{q}_{q}\right]^{\frac{1}{\alpha-q}},
	\end{equation}
	where $\alpha\in\{\ell,m\}$. The same property can be ensured for the Nehari manifold $\mathcal{N}_{\lambda}^{-}$, i.e, we have that $(u_n) \in \mathcal{N}_{\lambda}^{-}$ is bounded form above and below by positive constants.
\end{prop}
\nd \proof Remembering that $(u_n)\subset\N$, $m\Phi(t)\leq \phi(t)t^2$  and using the inequalities just above, we obtain that
\begin{equation}
\begin{array}{rcl}\label{12}
0> \alpha_\lambda^{+} + o_{n}(1) &>&J_\lambda(u_n)=\Int \Phi(|\nabla u_n|)\\[2ex]
&-&\displaystyle\frac{1}{\el}\phi(|\nabla u_n|) |\nabla u_n|^{2}-\lambda\left(\frac{1}{q}-\frac{1}{\el}\right)a(x)|u_n|^q\\[2ex]
&\geq& \Int\left(1-\Fr{m}{\el}\right) \Phi(|\nabla u_n|)-\lambda\left(\frac{1}{q}-\frac{1}{\el}\right)a(x)|u_n|^q
\end{array}
\end{equation}
holds for any $n\in\mathbb{N}$ large enough.  Under these conditions, using the above inequality and the continuous embedding $W_0^{1,\Phi}(\Omega)\hookrightarrow L^q(\Omega)$, we easily see that
$$0<-\left(\alpha_\lambda^++\Fr{1}{n}\right)\left[\Fr{\el q}{(\el-q)\lambda}\right]<\Int a(x)|u|^q\leq ||a||_\infty S^{q}_{q}||u_n||^q.$$
As a product the last estimate says that
\begin{equation*}\label{12a}
||u_n||> \left[-\left(\alpha_\lambda^++\Fr{1}{n}\right)\Fr{\el q}{(\el-q)\lambda\|a\|_\infty S^{q}_{q}}\right]^{\frac{1}{q}}.
\end{equation*}
As a consequence using the last estimate and Lemma \ref{nehari+} we see also that \eqref{12b} holds.

Furthermore, using $\eqref{12}$ and arguing as in the previous inequalities, we can also shown that
\begin{eqnarray}
\displaystyle \min\{||u_n||^\ell,||u_n||^m\}\leq \Int \Phi(|\nabla u_n|)&<& \lambda\left(\Fr{\el}{\el-m}\right)\left(\Fr{\el -q}{\el q}\right) \|a\|_\infty S^{q}_{q} ||u_n||^q \nonumber \\
&=& \Fr{\lambda}{q}\left(\Fr{\el-q}{\el-m}\right)\|a\|_\infty S^{q}_{q}||u_n||^q. \nonumber \\
\end{eqnarray}
Hence the last assertions give us
 $$\min\{||u_n||^{\ell-q},||u_n||^{m-q}\}<\lambda\left(\Fr{\el}{\el-m}\right)\left(\Fr{\el -q}{\el q}\right)\|a\|_\infty S^{q}_{q}= \Fr{\lambda}{q}\left(\Fr{\el-q}{\el-m}\right)\|a\|_\infty S^{q}_{q}.$$
 As a consequence we obtain \eqref{u_n-lim}.\hfill\cqd

Now we consider two technical results in order to prove that any minimizer sequence for $J$ on the Nehari manifold is a Palais-Smale sequence.

\begin{prop}\label{est-J'} Suppose $(\phi_{1})- (\phi_{3})$ and $(H)$.
	Then any minimizer sequence $(u_n)$ on the Nehari manifold $\mathcal{N}_{\lambda}^{-}$ or $\mathcal{N}_{\lambda}^{+}$  satisfies
	\begin{equation}\label{17d}
	\left <J'_\lambda(u_n), \Fr{u}{||u||}\right>\leq \Fr{C}{n}[||\xi'_n(0)||+1],~u\in W^{1,\Phi}(\Omega)/\{0\},
	\end{equation}
	where $\xi_n:=\xi:B_{\frac{1}{n}}(0)\rightarrow (0,\infty)$  was obtained by Lemma \ref{lem1ps}.
\end{prop}
\proof According to Lemma \ref{lem1ps}, we obtain
$$\xi_n:B(0,\epsilon_n) \to \mathbb{R}^+,\ \ \ \xi(0)=1,\,\ \xi(w)(u_n-w)\in\N^+.$$
Now, we put $\rho \in (0, \epsilon_n)$ and $u\in W^{1,\Phi}(\Omega)\backslash \{0\}$. Define the auxiliary function
$$w_\rho=\Fr{\rho u}{||u||}\in B(0,\epsilon_n).$$
Using one more time Lemma \ref{lem1ps} and $(ii)$ we mention that
\begin{equation}\label{177e}
\mu_\rho=\xi(w_\rho)(u_n-w_\rho)\in\N^+ \,\,\mbox{ and } \,\, J_\lambda(\mu_\rho)-J_\lambda(u_n)\geq- \frac{1}{n}||\mu_\rho-u_n||.
\end{equation}
Notice also that
 \begin{equation} \label{17c}
 w_\rho\to 0,\ \xi_n(w_\rho)\to 1, \ \mu_\rho\to u_n \mbox{ and } J'_\lambda(\mu_\rho)\to J'_\lambda(u_n)
 \end{equation}
as $\rho\to 0$ holds true for any $n \in \mathbb{N}$.

At this moment, applying Mean Value Theorem, there exists $t\in(0,1)$ in such way that
$$\begin{array}{rcl}J_\lambda(\mu_\rho)-J_\lambda(u_n)&=&\left<J'_\lambda((1-t)\mu_\rho+tu_n),\mu_\rho-u_n\right>\\[2ex]
&=&\left<J'_\lambda(\mu_\rho+t(u_n-\mu_\rho)) -J'_\lambda(u_n),\mu_\rho-u_n \right>\\[2ex]
&+&\left<J'_\lambda(u_n),\mu_\rho-u_n\right>.
\end{array}$$
It is worthwhile to mention that $||u_n-\mu_\rho||\to 0$ as $\rho \to 0$. Hence, using \eqref{177e} and \eqref{17c}, we easily see that
\begin{equation*}
-\frac{1}{n}||\mu_\rho-u_n||\leq \left<J'_\lambda(u_n),\mu_\rho-u_n \right> + o_{\rho}(||\mu_\rho-u_n||)
\end{equation*}
where $o_{\rho}(.)$ denotes a quantity that goes to zero as $\rho$ goes to zero. Taking into account that $\mu_\rho\in\N^+$ it follows that
\begin{equation*}\label{16}
-\frac{1}{n}||\mu_\rho-u_n|| +o_{\rho}(||\mu_\rho-u_n||)\leq \left<J'_\lambda(u_n),-w_\rho\right>+(\xi_n(w_\rho)-1)\left<J'_\lambda(u_n),u_n-w_\rho\right>.
\end{equation*}
Furthermore, using the fact that $\left<J_\lambda^{\prime}(\mu_\rho),\mu_\rho\right>=0$, we mention that
\begin{eqnarray*}
-\frac{1}{n}||\mu_\rho-u_n|| &\leq& o_{\rho}(||\mu_\rho-u_n||) - \rho\left<J'_\lambda(u_n),\frac{u}{||u||}\right> \nonumber \\
&+&(\xi_n(w_\rho)-1)\left<J'_\lambda(u_n)-J'_\lambda(\mu_\rho),u_n-w_\rho\right>. \nonumber \\
\end{eqnarray*}
As a consequence the last estimates and \eqref{17c} say that
 \begin{eqnarray}
 \left<J'_\lambda(u_n),\frac{u}{||u||}\right> &\leq& \Fr{||\mu_\rho-u_n||}{n\rho}+\Fr{o_{\rho}(||\mu_\rho-u_n||)}{\rho}
\nonumber \\
&+& \Fr{(\xi_n(w_\rho)-1)}{\rho}\left<J'_\lambda(u_n)-J'_\lambda(\mu_\rho),u_n-w_\rho\right>.  \nonumber
\end{eqnarray}
It is no hard to see that
\begin{equation}\label{17a}
||\mu_\rho-u_n||\leq \rho|\xi_n(w_\rho)|+|\xi_n(w_\rho)-1| \,||u_n||\mbox { and } \ds\lim_{\rho\to 0}\Fr{|\xi_n(w_\rho)-1|}{\rho}\leq||\xi'_n(0)||.
\end{equation}
The last inequality is justified due the fact that
$$\ds\lim_{\rho\to 0}\Fr{|\xi_n(w_\rho)-1|}{\rho}=\left<\xi'_n(0),\frac{u}{||u||}\right>\leq ||\xi'_n(0)||.$$
Therefore, using the fact that $(u_n)$ is bounded and \eqref{17a}, we infer that
$$\begin{array}{rcl}
\ds\lim_{\rho\to 0} \Fr{||\mu_\rho-u_n||}{\rho n}&\leq&\ds\lim_{\rho\to 0}\Fr{1}{n}\left[||\xi_n(w_\rho)||+\Fr{|\xi_n(w_\rho)-1|}{\rho}||u_n||\right]\\[2ex]
&\leq& \Fr{1}{n}\left[1+||\xi'_n(0)||\,\,||u_n||\right]\leq \Fr{C}{n}\left[1+||\xi'_n(0)||\right].
\end{array}$$

On the other hand, using the fact that $\Fr{\xi_n(w_\rho)-1}{\rho}$ and $\xi_n(w_\rho)$ are bounded for $\rho>0$ small enough, we easily see that
 $$\begin{array}{rcl}
\|\mu_\rho-u_n \|&=&|\rho|\left|\left|\Fr{\xi_n(w_\rho)-1}{\rho}u_n-\xi_n(w_\rho)\Fr{u}{||u||}\right|\right|\\[2ex]
&\leq& |\rho|\left[\left|\Fr{\xi_n(w_\rho)-1}{\rho}\right|||u_n||+|\xi_n(w_\rho)|\right].%\Fr{u_n}{||u_n||
\end{array}$$
Since $(u_n)$ is bounded there exists a constant $C>0$ in such way that
\begin{equation*}\label{17b}
\Fr{||\mu_\rho-u_n||}{\rho}\leq C[||\xi'_n(0)||+1].
\end{equation*}
Putting all these estimates together we employ that there exists a constant $C > 0$ which is independent in $\rho>0$ in such way that \eqref{17d} holds. This ends the proof. \hfill\cqd

Now we shall consider a technical result in order to get Palais-Smale sequences on the Nehari manifold $\mathcal{N}_{\lambda}^{+}$ or $\mathcal{N}_{\lambda}^{-}$.

\begin{prop}\label{xi_n-uni} Suppose $(\phi_{1})- (\phi_{3})$ and $(H)$. Then given any minimizer sequence $(u_n)$ on the Nehari manifold $\mathcal{N}_{\lambda}^{-}$ or $\mathcal{N}_{\lambda}^{+}$ we obtain
\begin{equation}
||\xi'_n(0)|| \leq C \, \, \mbox{for each} \,\, n \in \mathbb{N}
\end{equation}
where $C > 0$ is independent on $n$. Here we recall that $\xi_n:=\xi:B_{\frac{1}{n}}(0)\rightarrow (0,\infty)$ was obtained by Lemma \ref{lem1ps}.
\end{prop}

\proof Notice that the numerator in \eqref{ps} is bounded from below away zero by $ b ||v||$ where $b> 0$ is a constant. In order to prove the last assertion we shall consider some estimates. Initially, we define the auxiliary function $\chi_n : \w \rightarrow \mathbb{R}$ given by
$$\begin{array}{rcl}\chi_n(v)&=&\Int[\phi'(|\nabla u_n|)|\nabla u_n|+2\phi(|\nabla u_n|)]\nabla u_n\nabla v\\[2ex]&-&\el b(x)|u_n|^{\el-2}u_nv-q\lambda a(x)|u_n|^{q-2}u_nv.\end{array}$$
It is easy to verity that
\begin{equation*}\label{1a}
\begin{array}{rcl}|\chi_n(v)|&\leq&
\Int[ |\phi'(|\nabla u_n|)| |\nabla u_n|^2+2\phi(|\nabla u_n|)|\nabla u_n|]|\nabla v|\\[3ex]
&+&||b^+||_\infty\el\Int |u_n|^{\el-1}|v|+\lambda q||a^+||_\infty\Int |u_n|^{q-1}|v|.
\end{array}
\end{equation*}
Now using Remark \ref{conseqphi3} we see that $\Fr{|\phi'(t)t|}{\phi(t)}\leq \max\{|\ell-2|,|m-2|\}:=C_1$. Thus, using Holder's inequality, we also see that
$$\begin{array}{rcl}|\chi_n(v)|&\leq& C_1\Int\phi(|\nabla u_n|)|\nabla u_n||\nabla v|\\[2ex]
&+&||b^+||_\infty\el\Int |u_n|^{\el-1}|v|+\lambda q||a^+||_\infty\Int |u_n|^{q-1}|v|\\[2ex]
&\leq& 2C_1||\phi(|\nabla u_n|)|\nabla u_n|||_{\tilde{\Phi}}||v||\\[2ex]
&+&||b^+||_\infty\el\Int |u_n|^{\el-1}|v|+\lambda q||a^+||_\infty\Int |u_n|^{q-1}|v|\\[2ex]
&\leq& C_2 \max\left\{\left(\Int \tilde{\Phi}(\phi(|\nabla u_n|))|\nabla u_n|\right)^{\frac{\ell-1}{\ell}},\left(\Int \tilde{\Phi}(\phi(|\nabla u_n|))|\nabla u_n|\right)^{\frac{m-1}{m}}\right\} ||v|| \\[2ex]
&+&||b^+||_\infty\el\Int |u_n|^{\el-1}|v|+\lambda q||a^+||_\infty\Int |u_n|^{q-1}|v|.
\end{array}
$$
In view of inequality $\widetilde\Phi(t\phi(t))\leq \Phi(2t)\leq 2^m\Phi(t), t \geq 0$ and \eqref{u_n-lim} there exists a constant $C_3>0$ such that
$$\begin{array}{rcl}|\chi_n(v)|
&\leq& C_3\max\left\{\left(\Int \Phi(|\nabla u_n|)\right)^{\frac{\ell-1}{\ell}},\left(\Int \Phi(|\nabla u_n|)\right)^{\frac{m-1}{m}}\right\}||v||\\[2ex]
&+&||b^+||_\infty\el\Int |u_n|^{\el-1}|v|+\lambda q||a^+||_\infty\Int |u_n|^{q-1}|v|\\[2ex]
&\leq&
C_3 || u_n||^\beta||v||+||b^+||_\infty\el\Int |u_n|^{\el-1}|v|+\lambda q||a^+||_\infty\Int |u_n|^{q-1}|v|\\[2ex]
&\leq& C_4||v||+||b^+||_\infty\el\Int |u_n|^{\el-1}|v|+\lambda q||a^+||_\infty\Int |u_n|^{q-1}|v|.
\end{array}
$$ where $\beta\in\{\ell-1,\frac{\ell}{m}(\ell-1),m-1,\frac{m}{\ell}(m-1)\}$.

At this stage, we shall estimate the terms $\Int |u_n|^{\el-1}|v| $ and $\Int |u_n|^{q-1}|v|$.
In order to do that we employ Holder's inequality and Sobolev imbedding proving that
$$\Int |u_n|^{\el-1}|v|\leq \left(\Int|u_n|^{\el}\right)^{\frac{\el-1}{\el}}\left(\Int|v|^{\el}\right)^{\frac{1}{\el}}\leq C_5 \|u_n\|^{\ell^{*}-1} \|v\| \leq C_6 \|v\|.$$
In view of the estimates above there exists a constant $c>0$ in such that $|\chi_n(v)|\leq c||v||$. Here we emphasize that estimate \eqref{u_n-lim} says that $c$ is independent on $n \in \mathbb{N}$.

It remains to show that there exists a constant $d>0$, independent in $n$, in such way that $\gamma''_{u_n}(1)\geq d$. The proof follows arguing by contradiction assuming that $\gamma''_{u_n}(1) = o_n(1)$. It follows from \eqref{12b} that there exists $a_\lambda>0$ satisfying
 \begin{equation}\label{nozero}
 \ds\liminf_{n \rightarrow \infty}||u_n||\geq a_\lambda>0
\end{equation}

%%%%%%%%%%%%%%%%%%%%%%%%%%%%%%%%%%%%%%%%%%%%%%%%%%%%%%%%%%%%%%%%%%%%%%%%%%%%%%%%%%%%%%%%%%%%%%%%%%%%%%%%%%%%%%%%%%%%%%
%%%%%%%%%%%%%%%%%%%%%%%%%%%%%%%%%%%%%%%%%%%%%%%%%%%%%%%%%%%%%%%%%%%%%%%%%%%%%%%%%%%%%%%%%%%%%%%%%%%%%%%%%%%%%%%%%%%%%%%%%%%%%%%%%%%%%%%%%%%%%%%%%%%%%
At this moment we emphasize that $\gamma''_{u_n}(1) = o_n(1)$. Using \eqref{nehari} and \eqref{eq5} we deduce that
$$o_n(1)=\gamma''_{u_n}(1)=\Int (2-q)\phi(|\nabla u_n| )|\nabla u_n|^2+\phi'(|\nabla u_n|)||\nabla u|^3+(q-\el)b|u_n|^{\el}.$$
%Under hypothesis
Using \eqref{conseqphi3} and Sobolev embeddings we also mention that
\begin{eqnarray}
	(\ell-q)\Int \phi(|\nabla u| )|\nabla u_n|^2&\leq&(\el-q)||b^+||_{\infty}||u_n||_{\el}^{\el}+o_n(1)\nonumber\\
	&\leq& (\el-q)S_{\el}||b^+||_{\infty}||u_n||^{\el}+o_n(1).\nonumber
\end{eqnarray}

On the other hand, we observe that
$$(\ell-q)\Int \phi(|\nabla u_n| )|\nabla u_n|^2dx\geq\ell(\ell-q)\Int \Phi(|\nabla u_n|)dx\geq \ell(\ell-q)\min\{||u_n||^\ell,||u_n||^m\}. $$
Using the estimates just above we get
$$\ell(\ell-q)\min\{||u_n||^\ell,||u_n||^m\}\leq(\el-q)S_{\el}||b^+||_{\infty}||u_n||^{\el}+o_n(1).$$
Hence, we have that
$$\ell(\ell-q)\leq {(\el-q)S_{\el}||b^+||_{\infty}}||u_n||^{\el-\alpha}+\displaystyle\frac{o_n(1)}{||u_n||^\alpha}$$
where $\alpha=\ell$ whenever $||u_n||\geq 1$ and $\alpha=m$ whenever $||u_n||\leq 1$. Furthermore, using \eqref{nozero}, we obtain
\begin{equation}\label{des1aa}
||u_n||\geq \left[\Fr{\ell(\ell-q)}{(\el-q)S_{\el}||b^+||_{\infty}}\right]^{\frac{1}{\el-\alpha}}+o_n(1).
\end{equation}
Using one more time \eqref{conseqphi3} \eqref{eq5} and Holder inequality, we deduce that
\begin{eqnarray}
	(\el-m)\Int \phi(|\nabla u| )|\nabla u_n|^2&\leq&\lambda(\el-q)||a^+||_{(\frac{\ell}{q})'}||u_n||_{\ell}^{q}+o_n(1)\nonumber\\
	&\leq&\lambda(\el-q)S_{q}||a^+||_{(\frac{\ell}{q})'}||u_n||^{q}+o_n(1).\nonumber
\end{eqnarray}
Using the same ideas discussed here we also mention that
$$\ell(\el-m)\min\{||u_n||^\ell,||u_n||^m\}\leq \lambda(\el-q)S_{\ell}||a^+||_{(\frac{\ell}{q})'}||u_n||^{q}+o_n(1).$$
As a consequence we get
$$\Fr{\ell(\el-m)}{(\el-q)S_{\ell}||a^+||_{(\frac{\ell}{q})'}}||u_n||^\alpha=\Fr{\ell(\el-m)}{(\el-q)S_{\ell}||a^+||_{(\frac{\ell}{q})'}}\min\{||u_n||^\ell,||u_n||^m\}\leq \lambda||u_n||^{q}+o_n(1).$$
To sum up, using the estimate \eqref{nozero}, we can be shown that
\begin{eqnarray}
	||u_n||&\leq& \left[\lambda \Fr{(\el-q)S_{\ell}\|a^+\|_{\left(\frac{\ell}{q}\right)'}}{\ell(\el-m)}\right]^{\frac{1}{\alpha-q}}+\Fr{o_n(1)}{||u_n||^{\frac{q}{\alpha-q}}}\nonumber\\
	&=&\left[\lambda \Fr{(\el-q)S_{\ell}\|a^+\|_{\left(\frac{\ell}{q}\right)'}}{\ell(\el-m)}\right]^{\frac{1}{\alpha-q}}+o_n(1).\nonumber
\end{eqnarray}

%%%%%%%%%%%%%%%%%%%%%%%%%%%%%%%%%%%%%%%%%%%%%%%%%%%%%%%%%%%%%%%%%%%%%%%%%%%%%%%%%%%%%%%%%%%%%%%%%%%%%%%%%%%%%%%%%%%%%%%%%%%%%
%%%%%%%%%%%%%%%%%%%%%%%%%%%%%%%%%%%%%%%%%%%%%%%%%%%%%%%%%%%%%%%%%%%%%%%%%%%%%%%%%%%%%%%%%%%%%%%%%%%%%%%%%%%%%%%%%%%%%%%%%%%%%%
%Arguing as in the proof of Lemma \ref{c1} we infer that
%\begin{eqnarray}
%	||u_n|| &\geq& \left(\Fr{(\ell-q)\ell}{(\el-q)\|b\|_\infty S_{\el}}\right)^\frac{1}{\el-\alpha}+\Fr{o_n(1)}{||u_n||^{\frac{\alpha}{\el-\alpha}}}\nonumber\\
	%&=&\left(\Fr{( 		\ell-q)\ell}{(\el-q)\|b\|_\infty S_{\el}}\right)^\frac{1}{\el-\alpha}+o_n(1)\nonumber
%\end{eqnarray}
%holds true for any $\lambda\in (0,\lambda_1)$.
Arguing as in the proof of Lemma \ref{c1}, using the above inequality and \eqref{des1aa}, we have a contradiction for each $\lambda < \lambda_1$ where $\lambda _1$ was given by \eqref{Lambda1}. This finishes the proof. \hfill\cqd
\vskip.2cm

At this stage we shall prove that any minimizer sequences on the Nehari manifold in $\mathcal{N}^{+}_{\lambda}$ or $\mathcal{N}^{+}_{\lambda}$  provides us a Palais-Smale sequence. More specifically, we can prove the following result

\begin{prop}\label{3.1} Suppose $(\phi_{1})- (\phi_{3})$ and $(H)$.
Then we have the following assertions
\begin{enumerate}
\item there exists a sequence $(u_n)\subset \N$ such that
$J_\lambda(u_n)= \alpha^+_\lambda+o_n(1)
 \ \mbox{\and } J'_\lambda(u_n)=o_n(1)\ \mbox{ in}~  W^{-1,\widetilde{\Phi}}(\Omega).$
\item there exists a sequence $(u_n)\subset \N^-$ such that
$J_\lambda(u_n)=\alpha^-_\lambda+o_n(1)
 \ \mbox{\and } J'_\lambda(u_n)=o_n(1)\ \mbox{ in } W^{-1,\widetilde{\Phi}}(\Omega).$
\end{enumerate}
\end{prop}

\begin{proof}
Here we shall prove the item $(1)$. The proof of item $(2)$ follows the same ideas discussed here using Lemma \ref{lem2ps} instead of Lemma \ref{lem1ps}.  Applying Ekeland's variational principle there exists a sequence $(u_n)\subset\N^+$ in such  way that
\begin{description}
\item[(i)] $J_\lambda(u_n) \,= \,\alpha_\lambda^+ + o_n(1)$,
\item [(ii)] $J_\lambda(u_n)<J_\lambda(w)+\frac{1}{n}||w-u||, \,\,\forall \,\, w\,\in\N^+.$
\end{description}
According to Proposition \ref{xi_n-uni} there exists $C>0$ independent on $n \in \mathbb{N}$ in such way that $\|\xi_n(0)\|\leq C$. This estimate together with Proposition \ref{est-J'} give us the following estimate
$$\left<J'_\lambda(u_n), \Fr{u}{||u||}\right>\leq \Fr{C}{n},~u\in W^{1,\Phi}(\Omega)/\{0\}.$$
As a consequence $\|J'(u_n)\|\rightarrow0$ as $n \rightarrow \infty$. This ends the proof.
\end{proof}

\section{The concentration compactness method}

In this section we shall discuss the Concentration compactness Theorem for Orlicz-Sobolev framework. It is important to recover that compactness phenomena is a powerful property in variational methods. This property allow us to prove our main results on existence and multiplicity of solutions to quasilinear elliptic problem \eqref{eq1}.

In what follows we follow same ideas discussed in Willem \cite{Willem}. Given any function $v\in {C}^\infty_0(\Omega)$ we extend the function $v$ in the following form $v(x)=0$ for any $x\in\Omega^c$. This function is also denoted by $v$ which belongs to $v\in {C}_0^\infty(\mathbb{R}^N)$. Moreover, we observe that $\mbox{supp}(v)\subset \Omega $. It is important to mention also that
$$
\| v \|_{W^{1,\Phi}(\mathbb{R}^N)}=\| v\|_{W^{1,\Phi}(\Omega)}
$$
\nd and
$$
\|v\| = \|v\|_{_{W^{1,\Phi}(\Omega)}} \,\, \mbox{for any} \,\, v \in \w.
$$
\nd Furthermore, we observe that
$$
\w=\overline{\{v\in {C}_0^\infty(\mathbb{R}^N)~|~\mbox{supp}(v)\subset\Omega\}}^{W^{1,\Phi}(\mathbb{R}^N)}.
$$
\nd As a consequence we know that $v \in {W^{1,\Phi}(\mathbb{R}^N)}$ whenever $v\in\w$.
\par Now we shall consider the vectorial space
$$
{C}_0 =\overline{\{u\in {C}(\Omega)~|~\mbox{supp}(u) \buildrel \mbox{\scriptsize {cpt}} \over \subseteq \mathbb{R}^N\}}^{|\cdot|_\infty},
$$
\nd endowed with the norm $\displaystyle \|u\|_\infty=\sup_{x\in \mathbb{R}^N}|u(x)|$. Denote by $\mathcal{M}$ the space of finite measures $\mathbb{R}^N$ using the norm
$$
\|\mu\|_{\mathcal{M}}=\sup\left\{\int u d\mu~|~u\in {C}_0,~ |u|_\infty=1\right\}.
$$
Recall that $\mathcal{M}$ satisfies the following properties
	\begin{description}
		\item[(i)] $\mathcal{M} = {C}_0^*$  and $\langle \mu,u \rangle := \int u d\mu$,
		
		\item[(ii)] The convergence $\mu_n \buildrel \mathcal{M} \over \rightharpoonup \mu$  occurs whenever $\int u d\mu_n\stackrel{n\rightarrow\infty}{\longrightarrow} \int u d\mu,~u\in {C}_0$,
		
		\item[(iii)] Let $(\mu_n)\subseteq \mathcal{M}$ be an bounded sequence. Then, up to a subsequence, we obtain $\mu_n \buildrel \mathcal{M} \over \rightharpoonup \mu$.
	\end{description}

\nd At this moment we observe that any minimizer sequence $(u_n)\subseteq \mathcal{N}_{\lambda}$ is bounded.  Consider $\mu_n,\nu_n: {C}_0\rightarrow \mathbb{R}$ given by
$$
\langle\mu_n, v\rangle=\int_{\mathbb{R}^N}\Phi(|\nabla u_n|)vdx\qquad\mbox{and}\qquad\langle\nu_n, v\rangle=\int_{\mathbb{R}^N}|u_n|^{\ell^*}vdx,~v\in {C}_0.
$$
\nd Hence there exists a constant $C > 0$ in such way that
$$
|\langle\mu_n,v\rangle| \leq C \|v\|_\infty~\mbox{and}~  |\langle\nu_n,v\rangle|\leq C\|v\|_\infty.
$$
\nd In other words, we have been shown that $(\mu_n),(\nu_n)\subseteq \mathcal{M}$ are bounded measures. It follows from the last estimate that
\begin{equation}\label{conv_med}
\Phi(|\nabla u_n|)\rightharpoonup \mu,~~|u_n|^{\ell^*}\rightharpoonup \nu~~\mbox{in}~~\mathcal{M}.
\end{equation}
\nd In what follows we shall consider the Compactness-Concentration Theorem in the Orlicz-Sobolev framework, see Lions \cite{lions1}. For a simple demonstration on compactness-concentration theorem we refer the reader to Fukagai at. al \cite{Fuk_1}.

\begin{lem}\label{conc_comp}
	\nd  There exist an enumerable set $J$, a family $\{x_j\}_{j\in J}\subseteq \mathbb{R}^N$  such that $x_i \ne x_j$  and nonnegative real numbers $\{\nu_j\}_{j\in J}$ and $\{\mu_j \}_{j\in J}$ satisfying
	$$
	\nu=|u|^{\ell^*}+\sum_{j\in J}\nu_j\delta_{x_j}~\qquad\mbox{and}~\qquad\mu\geq \Phi(|\nabla u|)+\sum_{j\in J}\mu_j\delta_{x_j},
	$$
	\nd where $\delta_{x_j}$  is the Dirac measure with mass at $x_j$. Furthermore, we have
	$$
	\nu_j\leq\max\left\{S_{\ell^*}^{-\frac{\ell^*}{\ell}}\mu_j^{\frac{\ell^*}{\ell}},
	S_{\ell^*}^{-\frac{\ell^*}{m}}\mu_j^{\frac{\ell^*}{m}}\right\},~j\in J,
	$$
	where $S_{\ell^*}$ is the best constant for the embedding $W^{1,\Phi}_{0}(\Omega) \subset L^{\ell^{*}}(\Omega)$.
\end{lem}

\begin{lem}\label{J_finito-1}
	The set $J = \{j\in J~|~\nu_j>0\}$ is finite.
\end{lem}
\dem First of all, we mention that $\{x_j\}_{j\in \widetilde{J}}\subseteq \overline{\Omega}$. Indeed, arguing by contradiction we suppose that $x_j\in\overline{\Omega}^c$ for some $j\in J$. Hence there exists $\epsilon > 0$ such that ${\overline{B}_\epsilon}(x_j)\subseteq\overline{\Omega}^c$. Consider $\varphi_\epsilon\in {C}_0^\infty(\mathbb{R}^N)$ satisfying the following conditions
$$
supp(\varphi_\epsilon)\subset B_\epsilon(x_j),~~~
\varphi_\epsilon\stackrel{\epsilon\rightarrow 0}\longrightarrow \chi_{\{x_j\}}~ \mbox{a.e.}~ {\bf R}^{N}.
$$
\nd At this moment, we extend the function $u_n$ in $\mathbb{R}^N$ putting $u_n(x)=0$ for any $x\in\mathbb{R}^N-\Omega$. Let $\epsilon > 0$ be fixed. Using \eqref{conv_med}, we mention that
$$
0=\int_{\mathbb{R}^N}\Phi(|\nabla u_n|)\varphi_\epsilon dx\stackrel{n} \longrightarrow \int_{\mathbb{R}^N}\varphi_\epsilon d\mu.
$$
\nd Taking the limit as $\epsilon \to 0$ we deduce that
$$
0=\int_{\mathbb{R}^N}\varphi_\epsilon d\mu=\int_{B_\epsilon(x_j)}\varphi_\epsilon d\mu  \rightarrow\int_{\{x_j\}}d\mu=\mu_j.
$$
\nd As a consequence $\mu_j = 0$. According to Lemma \ref{conc_comp} we infer that $\nu_j=0$. This is a contradiction due the fact that $j\in J$. Hence
we obtain that $\{x_j\}_{j\in \widetilde{J}}\subseteq \overline{\Omega}$.

\par Consider $\psi\in C_0^\infty$ such that $0 \leq \psi \leq 1$, $\psi(x) = 1~ \mbox{if}~ |x| \leq 1$ and $\psi(x) = 0~\mbox{if}~|x| \geq 2$. Define $x_j$ with $j\in J$, $\epsilon > 0$ and
$$
\psi_\epsilon(x):=\psi\left(\frac{x-x_j}{\epsilon}\right),~~x\in\mathbb{R}^N.
$$
\nd Now we point out that
\begin{equation}\label{MM}
\int_{\Omega}\phi(|\nabla u_n|)\nabla u_n\nabla v -\lambda a(x)|u_n|^{q-2}u_nv-b(x)|u_n|^{\ell^*-2}u_nv=o_n(1),~v\in\w.
\end{equation}

\nd Using the fact that $(\psi_\epsilon u_n)\subseteq\w$ is bounded it follows from $(\ref{MM})$ that
\begin{equation}\label{d1}
\begin{array}{lll}
\displaystyle\int_{\Omega}\phi(|\nabla u_n|)\nabla u_n\nabla (\psi_\epsilon u_n) & = & \lambda\displaystyle\int_{\Omega}a(x)|u_n|^q\psi_\epsilon dx+\int_{\Omega}b(x)|u_n|^{\ell^*}\psi_\epsilon dx+o_n(1).
\end{array}
\end{equation}
\nd On the other hand, using the estimate $t^2\phi(t)\geq\Phi(t)$ we observe that \\
$\displaystyle\int_{\Omega}\phi(|\nabla u_n|)\nabla u_n\nabla (\psi_\epsilon u_n)dx= $
\begin{equation}\label{d2}
\begin{array}{lll}
& = & \displaystyle\int_{\Omega}u_n\phi(|\nabla u_n|)\nabla u_n\nabla \psi_\epsilon dx+\int_{\Omega}\psi_\epsilon\phi(|\nabla u_n|)|\nabla u_n|^2dx \\[2ex]
& \geq & \displaystyle\int_{\Omega}u_n\phi(|\nabla u_n|)\nabla u_n\nabla \psi_\epsilon dx+\int_{\Omega}\psi_\epsilon\Phi(|\nabla u_n|)dx.
\end{array}
\end{equation}
\par Now we claim that
\[
(\phi(|\nabla u_n|) |\nabla u_n|)~ \mbox{is bounded in}~ L_{\widetilde{\Phi}}(\Omega).
\]
In fact, using that $\widetilde \Phi(\phi(t)t)\leq \Phi(2t),~t\in \mathbb{R}$, we infer that
$$
\int_\Omega \widetilde \Phi(\phi(|\nabla u_n|)|\nabla u_n|)dx\leq\int_\Omega\Phi(2|\nabla u_n|)dx\leq 2^m\int_\Omega\Phi(|\nabla u_n|)dx<\infty.
$$
\nd This proves the claim. Hence, we have $(\phi(|\nabla u_n|){\partial u_n}/{\partial x_i})$ is also bounded in $L_{\widetilde{\Phi}}(\Omega)$. So that
\begin{equation}\label{e4}
\displaystyle \phi(|\nabla u_n|)\frac{\partial u_n}{\partial x_i}\rightharpoonup w_i~\mbox{in}~L_{\widetilde{\Phi}}(\Omega),~i=1,...,N.
\end{equation}
\nd  Now define $w =(w_1,...,w_N)$. It is no hard to see that
\begin{equation}\label{e6}
\int_\Omega (u_n\phi(|\nabla u_n|)\nabla u_n\nabla\psi_\epsilon-u~ w.\nabla \psi_\epsilon)dx=o_n(1).
\end{equation}
\nd Indeed, using H\"older's inequality and $(\ref{e4})$ in the testing function $\frac{\partial \psi_\epsilon}{\partial x_i}u$, we know that
$$
\begin{array}{cl}
\displaystyle  \Big |\int_\Omega  \phi(|\nabla u_n|)\frac{\partial u_n}{\partial x_i}\frac{\partial \psi_\epsilon}{\partial x_i}u_n-  w_i\frac{\partial \psi_\epsilon}{\partial x_i}u  dx \Big |~  \leq\\
\\
\displaystyle\int_\Omega \Big | \phi(|\nabla u_n|)\frac{\partial u_n}{\partial x_i}\frac{\partial \psi_\epsilon}{\partial x_i}\big(u_n- u \big) \Big | dx+
\Big |  \int_\Omega \phi(|\nabla u_n|)\frac{\partial u_n}{\partial x_i}\frac{\partial \psi_\epsilon}{\partial x_i}u - w_i\frac{\partial \psi_\epsilon}{\partial x_i}u  dx \Big| ~ \leq \\
\\
\displaystyle 2\left\|\phi(|\nabla u_n|)\frac{\partial u_n}{\partial x_i}\frac{\partial \psi_\epsilon}{\partial x_i}\right\|_{\widetilde{\Phi}}\left\|u_n-u\right\|_\Phi+ o_n(1).
\end{array}
$$
Here we have used that $u_n \rightharpoonup u$ in $\w$. Moreover, we mention that $\w\stackrel{comp}\hookrightarrow L_\Phi(\Omega)$ implies that $\|u_n-u\|_\Phi \rightarrow 0$. As a consequence we see that
\[
\int_\Omega \phi(|\nabla u_n|)\frac{\partial u_n}{\partial x_i}\frac{\partial \psi_\epsilon}{\partial x_i}u_ndx\stackrel{n}\longrightarrow\int_\Omega w_i\frac{\partial \psi_\epsilon}{\partial x_i}u dx,~~i=1,...,N,
\]
\nd This proves the assertion $(\ref{e6})$ proving the claim. Using $(\ref{e6})$ and $(\ref{d2})$ we get
\begin{equation}\label{d3}
\begin{array}{lll}
\displaystyle\int_{\Omega}\psi_\epsilon\Phi(|\nabla u_n|)dx
+\int_\Omega u w.\nabla \psi_\epsilon dx  & \leq &
\displaystyle\int_{\Omega}\phi(|\nabla u_n|)\nabla u_n\nabla (\psi_\epsilon u_n)+o_n(1).
\end{array}
\end{equation}
\nd It follows from $(\ref{d1})$ and $(\ref{d3})$ that
\[
\displaystyle\int_{\Omega}\psi_\epsilon\Phi(|\nabla u_n|)dx+\int_\Omega u w.\nabla \psi_\epsilon dx\leq \lambda\displaystyle\int_{\Omega}a(x)|u_n|^q\psi_\epsilon dx+\int_{\Omega}b(x)|u_n|^{\ell^*}\psi_\epsilon dx+o_n(1).
\]
\nd Taking the limit in the last assertion just above limit and using the fact that
$$
\int_{\Omega}\Phi(|\nabla u_n|)\psi_\epsilon dx\stackrel{n\rightarrow\infty}\longrightarrow \int_{\Omega}\psi_\epsilon d\mu,~~~ \int_{\Omega} b(x)|u_n|^{\ell^*}\psi_\epsilon dx\stackrel{n\rightarrow\infty}\longrightarrow \int_{\Omega}b(x)\psi_\epsilon d\nu,
$$
\nd and
$$
\int_{\Omega}a(x)|u_n|^q\psi_\epsilon dx\stackrel{n\rightarrow\infty}\longrightarrow\int_{\Omega}a(x)|u|^q \psi_\epsilon dx,
$$
\nd we deduce that
\begin{equation}\label{d5}
\displaystyle\int_{\Omega}\psi_\epsilon d\mu+\int_\Omega u w.\nabla \psi_\epsilon dx\leq \lambda\int_{\Omega}a(x)|u|^q\psi_\epsilon dx+\|b\|_\infty\int_{\Omega} \psi_\epsilon d\nu.
\end{equation}
\nd Let $v\in\w$ be fixed. Taking the limit in \eqref{MM} and using $(\ref{e4})$ we infer that
\begin{equation}\label{d6}
\int_\Omega\left(w.\nabla v-\lambda a(x)|u|^{q-2}uv -b(x)|u|^{\ell^*-2}u v\right)dx=0.
\end{equation}
\nd Putting $v=u \psi_\epsilon$ and $(\ref{d6})$ we obtain
\[
\int_\Omega u  w.\nabla \psi_\epsilon dx=\int_\Omega \left(\lambda a(x)|u|^q+b(x)|u|^{\ell^*}- w.\nabla u\right)\psi_\epsilon dx.
\]
\nd However, we observe that
$$
|\lambda a(x)|u|^q+b(x)|u|^{\ell^*}- w.\nabla u|\leq \lambda |a(x)||u|^q+|b(x)||u|^{\ell^*}+ |w||\nabla u|\in L^1(\Omega)
$$
\nd and
$$
\left(\lambda a(x)|u|^q+b(x)|u|^{\ell^*}- w.\nabla u\right)\psi_\epsilon\stackrel{\epsilon\rightarrow 0}\longrightarrow 0~\mbox{a. e. in}~\Omega
$$
\nd Using the Lebesgue Convergence Theorem we see that
$$
\int_\Omega  u w.\nabla \psi_\epsilon dx\stackrel{\epsilon\rightarrow 0}\longrightarrow 0\qquad\mbox{and}\qquad \int_\Omega b(x)|u|^{\ell^*} \psi_\epsilon dx\stackrel{\epsilon\rightarrow 0}\longrightarrow 0.
$$
\nd Note that
$$
\psi_\epsilon \stackrel{\epsilon\rightarrow 0}\longrightarrow \chi_{\{x_j\}}~\mbox{a.e. in}~ \mathbb{R}^N~\mbox{and}~ \psi_\epsilon(x)\leq \chi_{B_1(x_j)}(x)~\mbox{for any}~x \in \mathbb{R}^{N}
$$
for $\epsilon>0$ is small enough. In this case, we conclude that
$$
\int_{\mathbb{R}^N} \psi_\epsilon d \mu\stackrel{\epsilon\rightarrow 0}\longrightarrow \int_{\{x_j\}}d\mu=\mu(\{x_j\})=\mu_j~\mbox{and}~ \int_{\mathbb{R}^N} \psi_\epsilon d \nu\stackrel{\epsilon\rightarrow 0}\longrightarrow \int_{\{x_j\}}d\nu=\nu(\{x_j\})=\nu_j.
$$
\nd Taking the limit in $(\ref{d5})$ as $\epsilon\rightarrow 0_+$ we deduce that
\begin{equation}\label{d8}
\mu_j\leq \|b\|_\infty\nu_j,~~j\in J.
\end{equation}
\nd As a consequence Lemma \ref{conc_comp} we have $\mu_j\leq S^{\alpha}\mu_j^\alpha$ where $1<\alpha\leq\min \big\{{\ell^*}/{\ell},{\ell^*}/{m} \big \}$.
\nd Therefore, we obtain that $$\mu_j \geq \left(\frac{1}{ S_{\ell^*}^{\alpha}}\right)^{\frac{1}{\alpha-1}}.$$
To sum up, using $(\ref{d8})$ we easily see that
$$\nu_j\geq \frac{1}{\|b\|_\infty}\left(\frac{1}{ S_{\ell^*}^{\alpha}}\right)^{\frac{1}{\alpha-1}},~\mbox{for any}~j\in J.$$
At this stage, assuming that $\#( J) = \infty$ we obtain
$$
\sum_{j\in\widetilde{J}}\nu_j\geq\sum_{j\in\widetilde{J}}c_3=\infty.
$$
\nd This is impossible because of $\nu$ is a finite measure and
$$
\nu=|u|^{\ell^*}+\sum_{j\in \widetilde{J}}\nu_j\delta_{x_j}.
$$
\nd This finishes the proof of Lemma \ref{J_finito-1}. \hfill\cqd%$\hfill{\rule{2mm}{2mm}}$

For the next result we extend the function $u_n$ to ${\R}^{N}$ defining $u_n = 0~\mbox{in}~ {\R}^{N} \backslash \Omega$.

\begin{lem}\label{conv_compactos-1}
	Then there exist $r\in\mathbb{N}$ and $x_1, \cdots, x_r \in {\R}^{N}$ in such way that
	\begin{equation}\label{conv-K-1}
	u_n\stackrel{L^{\ell^*}(K)}\longrightarrow u
	\end{equation}
	for each set $K\subset \R^N \backslash \{x_1, \cdots , x_r \}$.
\end{lem}
\dem Initially, we observe that ${J}$ is finite. Hence there exists $\delta>0$ such that $B_{\delta}(x_j)\cap B_{\delta}(x_j)=\emptyset$ for any $i\neq j$ with $i, j \in {J}$. Now we consider the set $\displaystyle K_{\delta}\subset \R^N \backslash \cup_{j\in{J}} B_\delta(x_j)$ and
$\chi\in {C}_0^\infty$ in such way that
$$
0\leq\chi\leq 1,~~\chi=1~\mbox{in}~K_{\delta},~~\mbox{supp}(\chi)\cap\left(\cup_{j\in\widetilde{J}} B_{\frac{\delta}{2}}(x_j)\right)=\emptyset.
$$
\nd Notice also that
$$
|u_n-u|^{\ell^*}\rightharpoonup\nu~~\mbox{and}~~ \nu = \sum_{j\in\widetilde{J}}\nu_j\delta_{x_j}~\mbox{in}~ \mathcal{M}.
$$
\nd On the other hand, we observe that
$$
\displaystyle 0\leq\int_{K_{\delta}} |u_n-u|^{\ell^*}dx  \leq  \displaystyle\int_{\R^N}|u_n-u|^{\ell^*}\chi dx,
$$
$$
\displaystyle \int_{\R^N}|u_n-u|^{\ell^*}\chi dx \rightarrow \int_{\R^N}\chi d\nu,
$$
$$
\int_{\R^N}\chi d\nu  =
\displaystyle\sum_{j\in\widetilde{J}} \chi(x_j) = 0.
$$
\nd As a consequence we mention that
$$
\int_{K_{\delta}}|u_n-u|^{\ell^*}dx \rightarrow 0.
$$
\nd Putting the all estimates together and using the fact that $\delta > 0$ is arbitrary we conclude that \eqref{conv-K-1} holds true for each
compact set
$K\subset\R^N \backslash \{x_j\}_{j\in{J}}$. This ends the proof. \hfill\cqd%$\hfill{\rule{2mm}{2mm}}$

\begin{lem}\label{conv_grad_qtp}
	$(i)$ $\phi(|\nabla u_n|)\nabla u_n\rightharpoonup \phi(|\nabla u|)\nabla u$ em $\prod L_{\widetilde{\Phi}}(\Omega)$;\\
	$(ii)$ $|u_n|^{\ell^*-2}u_n \rightharpoonup|u|^{\ell^*-2}u$ em $L^{\frac{\ell^*}{\ell^*-1}}(\Omega)$.
\end{lem}
\dem Firstly, we shall prove the item $(i)$. Consider $\{K_\nu\}_{\nu =1}^{\infty}$ a family of compact sets satisfying
\begin{eqnarray}\label{rn-K_nu}
\Omega\backslash\{x_j\}_{j\in J}=\bigcup_{\nu=1}^\infty K_\nu.
\end{eqnarray}
Choose any integer number $\nu \geq 1$. Let $\chi\in {C}_0^\infty(\R^N)$ be a function such that
$0\leq\chi\leq 1$,~ $\chi=1~\mbox{in}~K_\nu$ and supp($\chi) \cap\{x_j\}_{j\in{J}} = \emptyset$. Using the fact that $\Phi$ is convex we obtain
$$A_n:=(\phi(|\nabla u_n|)\nabla u_n-\phi(|\nabla u|)\nabla u,\nabla u_n-\nabla u)\geq 0,~\mbox{in}~\R^N.$$
As a consequence we employ that
\begin{eqnarray}\label{monot}
0\leq \int_{K_\nu}A_n(x)dx & \leq & \int_{\R^n}(\phi(|\nabla u_n|)\nabla u_n,\nabla u_n-\nabla u)\chi dx\nonumber\\
&-&\int_{\R^n}(\phi(|\nabla u|)\nabla u,\nabla u_n-\nabla u)\chi dx.\nonumber
\end{eqnarray}

Define $v_n =\chi (u_n-u)$. It follows easily that $v_n$ is bonded in $W^{1,\Phi}(\R^N)$. Using $v_n$ as testing function in \eqref{MM} we deduce that
\begin{equation}\label{sol_eq}
\int_{\R^N}\phi(|\nabla u_n|)\nabla u_n\nabla v_ndx -\lambda\displaystyle\int_{\R^N}a(x)|u_n|^{q-2}u_nv_n -b(x)|u_n|^{\ell^*-2}u_nv_ndx=o_n(1).
\end{equation}
In other words, we know that
\begin{eqnarray}\label{eq-v_n}
\int_{\R^N}\phi(|\nabla u_n|)\nabla u_n(\nabla u_n-\nabla u) dx + \int_{\R^N}(u_n-u)\phi(|\nabla u_n|)\nabla u_n\nabla \chi dx\nonumber\\
= \displaystyle\int_{\Omega}(\lambda a(x)|u_n|^{q-2}u_n+ b(x)|u_n|^{\ell^*-2}u_n)v_ndx+o_n(1).
\end{eqnarray}
Note that
$$
\displaystyle \int_{\R^N} \big|\phi(|\nabla u_n|)\nabla u_n\nabla \chi(u_n- u) \big| dx \leq \displaystyle\|\phi(|\nabla u_n|)|\nabla u_n|\|_{\widetilde{\Phi}} |\nabla \chi|_\infty\|(u_n-u)\|_\Phi = o_n(1).
$$
Moreover, we mention that $L_{\Phi}(\Omega)\hookrightarrow L^\ell(\Omega)\hookrightarrow L^q(\Omega)$ which show that
\begin{eqnarray}
	\int_{\R^N}a(x)|u_n|^{q-1}|v_n|dx&\leq& \|a\|_\infty\|u_n\|_q^{q-1}\|u_n-u\|_q\nonumber\\
	&\leq& C \|a\|_\infty\|u_n\|_\Phi^{q-1}\|u_n-u\|_\Phi=o_n(1).\nonumber
\end{eqnarray}
\nd Additionally, using that $(|u_n|^{\ell^*-1})$ is bounded in $L^{\frac{\ell^*}{\ell^*-1}}(\Omega)$ and Lemma \ref{conv_compactos-1}, we conclude that
\[
\displaystyle \int_{\R^N}  |b(x)|| u_n|^{\ell^*-1}|v_n |dx  \leq  \|b\|_\infty\|u_n\|_{\ell^*}^{\ell^*-1}\|u_n-u\|_{L^{\ell^*}(S_{\chi})}
= o_{n}(1)
\]
where $S_\chi:=\mbox{supp}(\chi)$. In this way, using \eqref{eq-v_n} we get

\begin{eqnarray}\label{eq-u-0}
\int_{\R^N}\phi(|\nabla u_n|)\nabla u_n(\nabla u_n-\nabla u) \chi dx=o_n(1).	
\end{eqnarray}
Furthermore, using that $u_n\rightharpoonup u$ in $\w$ and $\chi \phi(|\nabla u|)|\nabla u|\in L_{\widetilde \Phi}(\Omega)$, putting $u_n=u=0$ in $\R^n \backslash \overline\Omega$ we see that
\begin{eqnarray}\label{eq-u}
\int_{\R^N}\phi(|\nabla u|)\nabla u(\nabla u_n-\nabla u)\chi dx=o_n(1).
\end{eqnarray}
At this stage using \eqref{eq-u-0}, \eqref{eq-u} in \eqref{monot} we ensure that $A_n$ is in $L^1(K_\nu)$. Now, up to a subsequence, we get
$$A_n(x)\rightarrow 0,~\mbox{a.e.}~x\in K_\nu.$$
Hence \eqref{rn-K_nu} implies that
$$A_n(x)\rightarrow 0,~\mbox{a.e.}~x\in \R^N.$$
It follows from \cite[Lemma 6]{DM} that
$$\nabla u_n\rightarrow \nabla u,~\mbox{a.e. in}~\R^N.$$
Moreover, using the fact that $u_n=0$ in $\R^N \backslash \overline{\Omega}$, we also see that
$$\nabla u_n\rightarrow \nabla u,~\mbox{a.e. in}~\Omega.$$
Using the fact that $t\longmapsto\phi(t)t$ is a continuous function one has
$$\phi(|\nabla u_n|)\nabla u_n \rightarrow \phi(|\nabla u|)\nabla u,~\mbox{a.e. in}~\Omega.$$
In this way, using that $\widetilde\Phi(\phi(t)t)\leq \Phi(2t)$, we obtain $\phi(|\nabla u_n|)|\nabla u_n|$ is bounded in $L_{\widetilde\Phi}(\Omega)$. Therefore, using \cite[Lem. 2, pg. 88]{gossez-Czech}, we have been shown that
$$\phi(|\nabla u_n|)\nabla u_n\rightharpoonup \phi(|\nabla u|)\nabla u,~\mbox{in}~\prod L_{\widetilde{\Phi}}(\Omega).$$
This ends the proof of item $i)$.

Now we shall prove the item $(ii)$. Note that $\w\stackrel{cpt}\hookrightarrow L_\Phi(\Omega)$ showing that $u_n\rightarrow u$ in $L_\Phi(\Omega)$. Up to a subsequence we have that $u_n\rightarrow u$ a. e. in $\Omega$. Hence we easily see that
$$|u_n|^{\ell^*-2}u_n \rightarrow|u|^{\ell^*-2}u,~\mbox{a.e. in }\Omega.$$
Now using the fact that $(|u_n|^{\ell^*-2}u_n) $ is bounded in $L^{\frac{\ell^*}{\ell^*-1}}(\Omega)$ and using one more time \cite[Lem. 2, pg. 88]{gossez-Czech} we conclude that
$$|u_n|^{\ell^*-2}u_n \rightharpoonup|u|^{\ell^*-2}u,~\mbox{in }L^{\frac{\ell^*}{\ell^*-1}}(\Omega).$$
This completes the proof. \hfill\cqd%$\hfill{\rule{2mm}{2mm}}$

\section{The proof of our main theorems}

\subsection{The proof of Theorem \ref{teorema1}}

Let $\lambda < \Lambda_1 = \displaystyle \min\{\lambda_1,\bar{\lambda}_1\}$ be fixed where $\lambda_1 > 0$ is given by \eqref{Lambda1} and $\bar{\lambda}_1>0$  is provided in \eqref{lambda1barra}. Taking into account Lemma \ref{nehari+} we infer that $$\alpha_\lambda^+:=\ds\inf_{u\in\N^+}J_\lambda(u) < 0.$$
The main feature here is to find a function $u=u_\lambda\in \N^+$ in such way that $$J_\lambda(u)= \ds\min_{u\in \N^+(\Omega)}J_\lambda (u)=:\alpha_\lambda^+ \,\, \mbox{and} \,\, J^{\prime}(u) \equiv 0.$$
As a first step, using Proposition \ref{lem1ps}, there exists a minimizer sequence denoted by $(u_n)\subset W^{1,\Phi}(\Omega)$ such that
\begin{equation}\label{cerami1}
J_\lambda(u_n)=\alpha_\lambda+o_n(1) \mbox{ and }
J'_\lambda(u_n)=o_n(1).
\end{equation}
Since the functional $J_\lambda$ is coercive in $\N^+$ we obtain that $(u_n)$  is now bounded in $\N^+$. Therefore, there exists a function $u\in{W^{1,\Phi}_0(\Omega)}$  in such way that
\begin{equation} \label{convergencia}
u_n \rightharpoonup u \,\, \mbox{ in } \,\, W_0^{1,\Phi}(\Omega),~~
u_n \to u \,\,\mbox{a.e.}\,\, \mbox{ in } \Omega,~~
u_n \to u \,\, \mbox{ in } \,\, L^{\Phi}(\Omega).
\end{equation}
\nd At this point we shall prove that $u$ is a weak solution for the problem elliptic problem \eqref{eq1}.
First of all, using \eqref{cerami1}, we mention that
$$\on=\left<J'_\lambda(u_n),v\right>=\Int \phi(|\nabla u_n|)\nabla u_n\nabla v-\lambda a(x)|u_n|^{q-2}u_n v-b(x)|u_n|^{\el-2}u_nv$$
holds for any $v \in \w$. In view of \eqref{convergencia} and Lemma \ref{conv_grad_qtp} we get
$$\Int \phi(|\nabla u|)\nabla u\nabla v-\lambda a(x)|u|^{q-2}u v-b(x)|u|^{\el-2}v = 0$$
for any $v\in W^{1,\Phi}(\Omega)$ proving that u is a weak solution to the elliptic problem \eqref{eq1}.
Additionally, the weak solution $u$ is not zero. In fact, using the fact that $u_n\in \mathcal{N}^{+},$ we obtain
$$\begin{array}{rcl}
\lambda\Int a(x)|u_n|^q&=& \Int( \Phi(|\nabla u_n|)- \Fr{1}{\el}\phi(|\nabla u_n|)|\nabla u_n|^2) \Fr{q\el}{\el-q} -J_\lambda(u_n) \Fr{q\el}{\el-q}\\[3ex]
&\geq& \Fr{q\el}{\el-q} \left(1 - \dfrac{m}{\ell^*}\right)\Int \Phi(|\nabla u_n|) -J_\lambda(u_n)\Fr{q\el}{\el-q} \\[3ex]
&\geq& -J_\lambda(u_n)\Fr{q\el}{\el-q}.
\end{array}
$$
Taking into account \eqref{cerami1} and \eqref{convergencia} we also obtain that
$$\lambda\Int a(x)|u|^q\geq-\alpha^{+}_\lambda\Fr{q\el}{\el-q} > 0.$$
As a consequence we deduce that $u\not\equiv 0$.

At this stage we shall prove that $J_\lambda(u)=\alpha_\lambda$ and $u_n\to u$ in $W_0^{1,\Phi}(\Omega)$.
Since $u\in \N$ we also see that
$$\alpha_\lambda \leq J_\lambda(u)=\Int \Phi(|\nabla u|)-\Fr{1}{\el}\phi(|\nabla u|)|\nabla u|^2-\lambda\left(\Fr{1}{q}-\Fr{1}{\el}\right)a(x)|u|^{q}.$$
Recall that
$$t\mapsto\Phi(t)-\Fr{1}{\el}\phi(t)t^2$$
is a convex function. In fact, using \eqref{conseqphi3}
and $m<\el$, we deduce that
\begin{eqnarray}
	\left(\Phi(t)-\Fr{1}{\el}\phi(t)t^2\right)''&=&\left[ \left(1-\frac{1}{\el}\right)t\phi(t)-\frac{1}{\el}t(t\phi(t))'\right]'\nonumber\\
	&=& (t\phi(t))' \left[\left(1-\frac{2}{\el}\right)-\frac{1}{\el}\frac{t(t\phi(t))''}{(t\phi(t))'}\right]\nonumber\\
	& \geq & (t\phi(t))' \left[\left(1-\frac{2}{\el}\right)-\frac{m-2}{\el}\right]\nonumber\\
	&=& (t\phi(t))'\left(1-\frac{m}{\el}\right)>0, t > 0.\nonumber
\end{eqnarray}
Hence the last assertion says that
$$u\longmapsto \Int \Phi(|\nabla u |)-\Fr{1}{\el}\phi(|\nabla u |)|\nabla u |^2dx$$
is weakly lower semicontinuous from below. Therefore we obtain
\begin{eqnarray}
	\alpha_\lambda \leq J(u) &\leq & \liminf \left(\Int \Phi(|\nabla u_n|)-\Fr{1}{\el}\phi(|\nabla u_n|)|\nabla u_n|^2\right.\nonumber\\
	&-&\left.\lambda\left(\Fr{1}{q}-\Fr{1}{\el}\right)a(x)|u_n|^{q}\right)\nonumber\\
	&=&\liminf J_\lambda(u_n)= \alpha_\lambda.\nonumber
\end{eqnarray}
As a consequence we have $J_\lambda(u)=\alpha_\lambda.$ Additionally, using \eqref{convergencia}, we also mention that
$$\begin{array}{rcl}
J_\lambda(u)&=&\Int \Phi(|\nabla u|)-\Fr{1}{\el}\phi(|\nabla u|)|\nabla u|^2-\lambda\left(\Fr{1}{q}-\Fr{1}{\el}\right)a(x)|u|^{q}\\[3ex]
&=&
\lim \left(\Int \Phi(|\nabla u_n|)-\Fr{1}{\el}\phi(|\nabla u_n|)|\nabla u_n|^2-\lambda\left(\Fr{1}{q}-\Fr{1}{\el}\right)a(x)|u_n|^{q}\right)\\[3ex]
&=&
\lim \left(\Int \Phi(|\nabla u_n|)-\Fr{1}{\el}\phi(|\nabla u_n|)|\nabla u_n|^2\right)-\lambda\left(\Fr{1}{q}-\Fr{1}{\el}\right)\Int a(x)|u|^{q}.
\end{array}$$
It follows from the last identity that
$$\lim \left(\Int \Phi(|\nabla u_n|)-\Fr{1}{\el}\phi(|\nabla u_n|)|\nabla u_n|^2\right)=\Int \Phi(|\nabla u|)-\Fr{1}{\el}\phi(|\nabla u|)|\nabla u|^2.$$
In view of Brezis-Lieb Lemma, choosing $v_n=u_n-u,$ we infer that
\begin{eqnarray}
	\lim \left(\Int \Phi(|\nabla u_n|)-\Fr{1}{\el}\phi(|\nabla u_n|)|\nabla u_n|^2+ \Phi(|\nabla v_n|)-\Fr{1}{\el}\phi(|\nabla v_n|)|\nabla v_n|^2\right)\nonumber\\
	=\Int \Phi(|\nabla u|)-\Fr{1}{\el}\phi(|\nabla u|)|\nabla u|^2.
\end{eqnarray}
 In this way, the previous assertion implies that
$$0=\lim \left(\Int \Phi(|\nabla v_n|)-\Fr{1}{\el}\phi(|\nabla v_n|)|\nabla v_n|^2\right)\geq \lim\left(1-\Fr{m}{\el}\right)\Int \Phi(|\nabla v_n|)\geq 0.$$ Therefore, we obtain that $\lim \int_{\Omega} \Phi(|\nabla v_n|)=0$ and $u_n\to u \,\,\mbox{in} \,\, W^{1,\Phi}(\Omega).$  Hence we conclude that $u_{n} \rightarrow u$ in $\w$.

At this point we shall ensure that $u\in \N^+$. Arguing by contradiction we have that $u\notin\N^+$. Using Lemma \ref{fib} there are unique $t_0^+, t_0^->0$ in such way that $t_0^+u\in\N^+$ and $t_0^-u\in\N^-$. In particular, we know that $t_0^+< t_0^-=1.$ Since
$$\Fr{d}{dt}J_\lambda(t_0^+u)=0 \,\,\mbox{and} \,\, \Fr{d^2}{dt^2}J_\lambda(t_0^+u)>0$$
there exist $t^- \in ( t_0^+, t_0^- )$ such that $J_\lambda(t_0^+u)<J_\lambda(t^-u)$. As a consequence $J_\lambda(t_0^+u)<J_\lambda(t^-u)\leq J_\lambda(t_0^-u)=J_\lambda(u)$ which is a contradiction due the fact that $u$ is a minimizer in $\N^+$. So that $u$ is in $\N^+$.

Due the fact that $J_\lambda(u)=J_\lambda(|u|)$ and $J^{\prime}_\lambda(u) =J^{\prime}_\lambda(|u|)$ we show that $|u|\in \N^+$ for each $u \in \N^{+}$. Taking into account Lemma \ref{criticalpoint} we conclude $|u|$ is also a critical point of $J_\lambda$. To sum up,  we assume that $u\geq 0.$ holds true.

Finally, we observe that $\displaystyle\lim_{\lambda\to 0}||u||=0$. Indeed, since $u\in\N^+$ and arguing as in the proof of Lemma \ref{c1}, we get
$$||u||^{\alpha-q}\leq \lambda\Fr{\el-q}{\ell(\el-m)}S_{\ell}||a^+||_{(\frac{\ell}{q})'}$$ where we put $\alpha \in \{\ell,m\}$. This ends the proof of Theorem \ref{teorema1}. \hfill\cqd

\subsection{The proof of Theorem \ref{teorema2}}

Put $\Lambda_2 = \min\{\bar{\lambda}_1,\tilde{\lambda}_1\}$ where $\bar{\lambda}_1$ is provided in \eqref{lambda1barra} and  $\tilde{\lambda}_1$ is given by Lemma \ref{nehari-}. Initially, due Lemma \ref{nehari-}, there exists $\delta_1>0$ such that $J_\lambda(v)\geq \delta_1$ for any $v\in \N^{-}.$
As a consequence
$$\alpha^{-}_{\lambda}:= \ds \inf_{v \in \N^{-}}J_\lambda(v)\geq \delta_1>0.$$

Now we shall consider a minimizer sequence $(v_n)\subset \N^{-}$ given in Proposition \ref{lem1ps}, i.e, $(v_n)\subset \N^{-}$ is a sequence satisfying
\begin{equation} \label{e1}
\ds\lim_{n\to\infty}J_\lambda(v_n)=\alpha_{\lambda}^{-} \,\,\mbox{and} \,\, \ds\lim_{n\to\infty} J^{\prime}(v_{n}) = 0.
\end{equation}
Since $J_\lambda$ is coercive in $\N$ and so on $\N^{-}$, using Lemma \ref{c1}, we can be shown that $(v_n)$ is a bounded sequence
in $W^{1,\Phi}_{0}(\Omega).$ Up to a subsequence we assume that $v_n\rightharpoonup v$ in $W^{1,\Phi}_{0}(\Omega)$ holds for some $v \in \w$.  Additionally, using the fact that $q <\ell^{*}$, we obtain $t^{q}<<\Phi_{*}(t)$ and $W_{0}^{1,\Phi}(\Omega)\hookrightarrow L^{q}(\Omega)$ is also a compact embedding. This fact ensures that $v_n\to v$ in  $L^{q}(\Omega).$ In this way, we easily seen that
\begin{equation*}\label{lim1}
\ds\lim_{n\to\infty} \Int a(x) |v_n|^{q}=\Int a(x) |v|^{q}.
\end{equation*}

Now we claim that $v \in \w$ given just above is a weak solution to the elliptic problem \eqref{eq1}. In fact, using \eqref{e1}, we infer that
$$\left<J'_\lambda(v_n), w\right>=\Int \phi(|\nabla v_n|)\nabla v_n\nabla w-\lambda a(x)|v_n|^{q-2}v_n w-b(x)|v_n|^{\el-2}v_n w = o_{n}(1)$$
holds for any $w \in \w$. Now using Lemma \ref{conv_grad_qtp} we get
$$\Int \phi(|\nabla v|)\nabla v\nabla w-\lambda a(x)|v|^{q-2}v w-b(x)|v|^{\el-2}v w = 0, w \in \w.$$
So that $v$ is a critical point for the functional $J_{\lambda}$.
Without any loss of generality, changing the sequence $(v_{n})$ by $(|v_{n}|)$, we can assume that $v \geq 0$ in $\Omega$.

Now we claim that $v \neq 0$. The proof for this claim follows arguing by contradiction assuming that $v \equiv 0$. Recall that
$J(t v_{n}) \leq J(v_{n})$ for any $t \geq 0$ and $n \in \mathbb{N}$. These facts imply that
\begin{eqnarray}
\left(1 - \dfrac{m}{\ell^{*}}\right)\int_{\Omega} \Phi(|\nabla t v_{n}|) &\leq& \lambda \left(t^{q} - 1\right)\left(\Fr{1}{q}-\Fr{1}{\el}\right) \int_{\Omega} a(x)|v_{n}|^{q}\nonumber \\
&+& \left(1 - \dfrac{\ell}{\ell^{*}}\right)\int_{\Omega} \Phi(|\nabla v_{n}|). \nonumber
\end{eqnarray}
Using the last estimate together with the fact that  $(v_{n})$ is bounded and \cite[Lemma 2.1]{Fuk_1}, we obtain
\begin{equation*}
\min(t^{\ell}, t^{m}) \left(1 - \dfrac{m}{\ell^{*}}\right) \int_{\Omega} \Phi(|\nabla v_{n}|) \leq \lambda \left(t^{q} - 1\right)\left(\Fr{1}{q}-\Fr{1}{\el}\right) \int_{\Omega} a(x)|v_{n}|^{q} + C
\end{equation*}
holds for some $C > 0$. These inequalities give us
\begin{equation*}
\min(t^{\ell}, t^{m}) \left(1 - \dfrac{m}{\ell^{*}}\right) \int_{\Omega} \Phi(|\nabla v_{n}|) \leq \lambda \left(t^{q} - 1\right)\left(\Fr{1}{q}-\Fr{1}{\el}\right) \|a\|_{\infty} \|v_{n}\|_{q}^{q} + C.
\end{equation*}
It is no hard to verify that the fact $\|v_{n}\| \geq c > 0$ for any $n \in \mathbb{N}$. Using one more time \cite[Lemma 2.1]{Fuk_1} we infer that
\begin{equation*}
 \min(t^{\ell}, t^{m}) \leq o_{n}(1) t^{q} + C
\end{equation*}
holds for any $t \geq 0$ where $C = C(\ell,m,\ell^{*}, \Omega, a,b) > 0$ where $o_{n}(1)$ denotes a quantity that goes to zero as $n \rightarrow \infty$. Here was used the fact $v_{n} \rightarrow 0$ in $L^{q}(\Omega)$. This estimate does not make sense for any $t > 0$ big enough using the fact that $q < \ell$. Hence $v \neq 0$ as claimed. As a consequence $v$ is in $\mathcal{N_{\lambda}} = \mathcal{N_{\lambda}^{+}} \cup \mathcal{N_{\lambda}^{+}}$.

At this stage we shall prove that $v_n\to v$ in $W_0^{1,\Phi}(\Omega)$. The proof follows arguing by contradiction.
Assume that $\displaystyle \liminf_{n \rightarrow \infty} \int_{\Omega} \Phi(|\nabla v_{n} - \nabla v|) \geq \delta$ holds for some $\delta > 0$.
Recall that $\Psi: \mathbb{R} \rightarrow \mathbb{R}$ given by
$$t\mapsto \Psi(t) = \Phi(t)-\Fr{1}{\el}\phi(t)t^2$$
is a convex function for each $t \geq 0$. The Brezis-Lieb Lemma for convex functions says that
\begin{equation*}
\lim_{n \rightarrow \infty} \int_{\Omega} \Psi(|\nabla v_{n}|) - \Psi(|\nabla v_{n}- v|) =  \int_{\Omega} \Psi(|\nabla v|)
\end{equation*}
In particular, the last estimate give us
\begin{equation*}
\int_{\Omega} \Psi(|\nabla v|) < \liminf_{n \rightarrow \infty}  \int_{\Omega} \Psi(|\nabla v_{n}|).
\end{equation*}
Since $v\in \N$ there exists unique $t_{0}$ in $(0, \infty)$ such that $t_{0} v \in \mathcal{N}_{\lambda}^{-}$. It is easy to verify that
\begin{equation*}
\int_{\Omega} \Psi(|\nabla t_{0}v|) < \liminf_{n \rightarrow \infty}  \int_{\Omega} \Psi(|\nabla t_{0} v_{n}|).
\end{equation*}
As a consequence we see that
\begin{eqnarray}
\alpha^{-}_\lambda &\leq& J_\lambda(t_{0}v ) = \Int \Psi(|\nabla t_{0} v|)-\lambda\left(\Fr{1}{q}-\Fr{1}{\el}\right)a(x)|t_{0} v|^{q} \nonumber \\
 &<& \liminf_{n \rightarrow \infty} \Int \Psi(|\nabla t_{0} v_{n}|)-\lambda\left(\Fr{1}{q}-\Fr{1}{\el}\right)a(x)|t_{0} v_{n}|^{q} \nonumber \\
 &=& \liminf_{n \rightarrow \infty} J_{\lambda}(t_{0}v_{n}) \leq \liminf_{n \rightarrow \infty} J_{\lambda}(v_{n}) = \alpha_{\lambda}^{-}.  \nonumber
\end{eqnarray}
This is a contradiction proving that $v_n\to v$ in $W_0^{1,\Phi}(\Omega)$. As a consequence $v$ is in $\mathcal{N}^{-}_{\lambda}$. This follows from the strong convergence and the fact that $t = 1$ is the unique maximum point for the fibering map $\gamma_{v}$ for any $v \in \mathcal{N}_{\lambda}^{-}$.
Hence using the same ideas discussed in the proof of Theorem \ref{teorema1} we infer that
\begin{equation*}
\alpha^{-}_\lambda \leq J_{\lambda}(v) \leq \liminf J_{\lambda}(v_{n}) = \alpha^{-}_\lambda.
\end{equation*}
In particular, we see that $\alpha^{-}_\lambda = J_{\lambda}(v)$ and
\begin{equation*}
\lim \int_{\Omega} \Phi(|\nabla v_{n}|) - \dfrac{1}{\ell^{*}} \phi(|\nabla v_{n}|)|\nabla v_{n}|^{2} = \int_{\Omega} \Phi(|\nabla v|) - \dfrac{1}{\ell^{*}} \phi(|\nabla v|)|\nabla v|^{2}.
\end{equation*}
In particular, we know that $J_{\lambda}(v)\geq \delta_{1} > 0$. So we finish the proof of Theorem \ref{teorema2}. \hfill\cqd

\subsection{The proof of Theorem \ref{teorema3}}

In view of Theorems \ref{teorema1} and \ref{teorema2} there are $u\in \mathcal{N}^{+}$ and $v \in \mathcal{N}^{-}$ in such way that $$J_\lambda(u)=\ds\inf_{w\in \mathcal{N}^{+}}J_\lambda(w)\ \ \ \mbox{and}\ \ \  J_\lambda(v)=\ds\inf_{ w\in \mathcal{N}^{-}}J_\lambda(w).$$
Additionally, using the fact that $0 < \lambda < \Lambda := \min\{\Lambda_1,\Lambda_2\}$ where $\Lambda_1, \Lambda_2>0$ are given by Theorem \ref{teorema1} and Theorem \ref{teorema2} we stress that $\N^{+}\cap \N^{-}= \emptyset$. Therefore, $u,v$ are nonnegative ground state solutions to the elliptic problem \eqref{eq1}. As was mentioned before, using the fact that
$$ J_\lambda(w)=J_\lambda(|w|)\,\, \mbox{and}\,\, J^{\prime}_\lambda(w)=J^{\prime}_\lambda(|w|)$$
holds true for any $w \in \w$ we can assume $u, v\geq 0$ in $\Omega$. Furthermore, $u$ and $v$ are nontrivial critical points for $J_{\lambda}$ proving that problem \eqref{teorema1} admits at least two nontrivial solutions whenever $0 < \lambda < \Lambda$.
This completes the proof.
\hfill\cqd

\end{document}